\documentclass[a4paper,
               headexclude,
               headsepline,
               footexclude,
               footsepline,
               twoside,
               abstracton,
               normalheadings,
               11pt]{scrartcl}
\usepackage{scrpage2}
\usepackage{amsmath,amsfonts,amsthm,amssymb,bm}
\usepackage[graph,frame,all]{xy}
\usepackage{color}
\usepackage{pxfonts}
\usepackage{mathpazo}
\usepackage{paralist}
\renewenvironment{enumerate}{\begin{compactenum}}{\end{compactenum}}
\renewenvironment{itemize}{\begin{compactitem}}{\end{compactitem}}
\renewenvironment{description}{\begin{compactdesc}}{\end{compactdesc}}
\newcounter{rememberEnumi}
\newcommand{\SaveEnumi}{\global\setcounter{rememberEnumi}{\value{enumi}}}
\newcommand{\RecallEnumi}{\setcounter{enumi}{\therememberEnumi}}
\addtokomafont{title}{\rmfamily}
\addtokomafont{section}{\rmfamily}
\addtokomafont{minisec}{\rmfamily}%% for "abstract"
\swapnumbers
\theoremstyle{plain}
\newtheorem{theo}{Theorem}[section]
\newtheorem{lemm}[theo]{Lemma}
\newtheorem{coro}[theo]{Corollary}
\newtheorem{prop}[theo]{Proposition}
\newtheorem{namet}[theo]{\myThmName}

\newcommand{\comments}{\par\medskip
  \noindent{\itshape Comments on \arabic{section}.\arabic{theo}}. }
\theoremstyle{definition} 

\newtheorem{defs}[theo]{Definitions} 
\newtheorem{nota}[theo]{Notation}
\newtheorem{exam}[theo]{Example}
\newtheorem{exas}[theo]{Examples} 

\newtheorem{prob}[theo]{Problem} 
\newtheorem{probs}[theo]{Problems}
\newtheorem{rema}[theo]{Remark} 
\newtheorem{rems}[theo]{Remarks}
\newtheorem*{ackn}{Acknowledgement}
\newtheorem{named}[theo]{\myThmName} \newtheorem*{named*}{\myThmName}
\newenvironment{ndef}[1][\kern-.235em]{\edef\myThmName{#1}\begin{named}}{\end{named}}
\newenvironment{ndef*}[1][\kern-.235em]{\edef\myThmName{#1}\begin{named*}}{\end{named*}}
\def\cprime{$'$}
\let\setminus\smallsetminus
\newcommand{\set}[2]{\left\{{#1}\left|\vphantom{#1#2\strut}\right.\, 
                    {#2}\right\}}
\newcommand{\smallset}[2]{\{{#1}\left|\vphantom{}\right.\, 
                    {#2}\}}
\newcommand{\Aut}[1]{\mathrm{Aut}(#1)}

\newcommand{\Hom}[3][]{\mathrm{Hom}_{#1}(#2,#3)}

\newcommand{\Z}[1]{\mathrm{Z}(#1)}
\newcommand{\C}[2][]{\mathrm{C}_{#1}(#2)}

\newcommand{\wt}{\mathop{\mathfrak{w}}}
\newcommand{\dual}[1]{#1^*}
\newcommand{\complete}[1]{\widehat{#1}}
\newcommand{\lie}[1]{{\mathfrak{L}}(#1)}
\newcommand{\comm}[2]{\operatorname{comm}\left(#1,#2\right)}
\newcommand{\ind}[2]{\mathrm{ind}_{#1}^{#2}}

\newcommand{\Ad}[1][]{\mathrm{Ad}_{#1}}
\newcommand{\CC}{\mathbb C}
\newcommand{\FF}{\mathbb F}
\newcommand{\NN}{\mathbb N}
\newcommand{\PP}{\mathbb P}
\newcommand{\QQ}{\mathbb Q}
\newcommand{\RR}{\mathbb R}
\newcommand{\TT}{\mathbb T}
\newcommand{\ZZ}{\mathbb Z}
\newcommand{\gL}[2][]{\mathrm{\mathfrak{gl}}_{#1}{#2}}
\newcommand{\GL}[2][]{\mathrm{GL}_{#1}{#2}}
\newcommand{\SL}[2][]{\mathrm{SL}_{#1}{#2}}
\newcommand{\sL}[2][]{\mathrm{\mathfrak{sl}}_{#1}{#2}}
\newcommand{\SU}[2][]{\mathrm{SU}_{#1}{#2}}

\newcommand{\sP}[2][]{\mathrm{\mathfrak{sp}}_{#1}{#2}}
\newcommand{\PSL}[2][]{\mathrm{PSL}_{#1}{#2}}
\newcommand{\PGL}[2][]{\mathrm{PGL}_{#1}{#2}}
\newcommand{\Orth}[2]{\mathrm{O}_{#1}{#2}}

\newcommand{\SO}[2]{\mathrm{SO}_{#1}{#2}}
\newcommand{\Spin}[1]{\mathrm{Spin}_{#1}}
\newcommand{\id}{\mathrm{id}}
\newcommand{\OlM}[1]{\mathrm{Olm}_{#1}}%% Olshanskii monster
\newcommand{\OR}[1]{\mathrm{OR}(#1)}
\newcommand{\KOname}{Kor}
\newcommand{\KOf}{{\text{\sc\KOname}}}
\newcommand{\KO}[1]{{\text{\sc\KOname}}(#1)}
\newcommand{\ko}[1]{{\mathfrak{\KOname}}(#1)}

\newcommand{\fP}{\mathbf{P}}
\newcommand{\aP}{\bm{\Pi}}
\newcommand{\hQ}{\mathbf{{Q}}}
\newcommand{\cQ}{\mathbf{\complete{Q}}}

\newcommand{\cS}{\mathbf{\overline{S}}}
\newcommand{\aS}{\mathbf{S}}
\newcommand{\ProLie}{\text{\sc ProLie}}
\newcommand{\LCG}{\text{\sc LCG}}
\newcommand{\LCA}{\text{\sc LCA}}
\newcommand{\CgAL}{\text{\sc CgAL}}
\newcommand{\Lie}{\text{\sc Lie}}
\newcommand{\SepLie}{\text{\sc SepLie}}
\newcommand{\TG}{\text{\sc TG}}
\newcommand{\CG}{\text{\sc CG}}
\newcommand{\CA}{\text{\sc CA}}
\newcommand{\Conn}[1]{\text{\sc Conn}{#1}}
\newcommand{\almConn}[1]{\text{\sc almConn}{#1}}
\newcommand{\Ab}[1]{\text{\sc Ab}{#1}}

\newcommand{\Small}[1]{\text{\sc small}{#1}}
\newcommand{\cMon}{\text{\sc monCA}}
\makeatletter
\newcommand{\widebar}[1]{\mathop {\mathchoice
    {\vbox
    {\m@th \ialign {##\crcr \noalign {\kern 1\p@ }\hrulefill \crcr
        \noalign {\kern 1\p@ \nointerlineskip }%
        $\hfil \textstyle {#1}\hfil $\crcr }}}
    {\vbox
    {\m@th \ialign {##\crcr \noalign {\kern 1\p@ }\hrulefill \crcr
        \noalign {\kern 1\p@ \nointerlineskip }%
        $\hfil \textstyle {#1}\hfil $\crcr }}}
    {\vbox
    {\m@th \ialign {##\crcr \noalign {\kern 1\p@ }\hrulefill \crcr
        \noalign {\kern 1\p@ \nointerlineskip }%
        $\hfil \scriptstyle {#1}\hfil $\crcr }}}
    {\vbox
    {\m@th \ialign {##\crcr \noalign {\kern 1\p@ }\hrulefill \crcr
        \noalign {\kern 1\p@ \nointerlineskip }%
        $\hfil \scriptscriptstyle {#1}\hfil $\crcr }}}%
    }}
\makeatother
\let\closure\widebar
\lohead[]{\footnotesize\rm \MYtitle}
\lehead[]{\footnotesize\rm \MYauthor}
\rehead[]{\footnotesize\rm \MYtitle}
\rohead[]{\footnotesize\rm \MYauthor}
\lefoot[\pagemark]{\pagemark}
\cefoot[]{}
\refoot[]{}
\lofoot[]{} 
\cofoot[]{}
\rofoot[\pagemark]{\pagemark}
\RequirePackage{hyperref} 
\title{Kernels of Linear Representations\\ of Lie
  Groups, Locally Compact Groups,\\ and Pro-Lie Groups }

\author{Markus Stroppel}
\makeatletter
  \let\MYauthor\@author 
  \let\MYtitle\shorttitle
\makeatother

\begin{document}
\maketitle

\begin{abstract}\noindent
  For a topological group $G$ the intersection $\KO{G}$ of all kernels
  of ordinary representations is studied. We show that $\KO{G}$ is
  contained in the center of~$G$ if $G$ is a connected pro-Lie group.
  The class $\KO{\mathcal{C}}$ is determined explicitly if
  $\mathcal{C}$ is the class $\Conn\Lie$ of connected Lie groups or
  the class $\almConn\Lie$ of almost connected Lie groups: in both
  cases, it consists of all compactly generated abelian Lie groups.
  Every compact abelian group and every connected abelian pro-Lie
  group occurs as $\KO{G}$ for some connected pro-Lie group~$G$.
  However, the dimension of $\KO{G}$ is bounded by the cardinality of
  the continuum if $G$ is locally compact and connected.  Examples are
  given to show that $\KO{\mathcal{C}}$ becomes complicated if
  $\mathcal{C}$ contains groups with infinitely many connected
  components.
\end{abstract}

\section{The questions we consider and the answers that we have found}

In the present paper we study (Hausdorff) topological groups. If all
else fails, we endow a group with the discrete topology.
\enlargethispage{5mm}

For any group $G$ one tries, traditionally, to understand the group by
means of representations as groups of matrices. To this end, one
studies the continuous homomorphisms from $G$ to $\GL[n]{\CC}$ for
suitable positive integers~$n$; so-called \emph{ordinary
  representations}. 
This approach works perfectly for finite groups because any such group
has a faithful ordinary representation but we may face
difficulties for infinite groups; there do exist groups admitting no
ordinary representations apart from the trivial (constant)
one. See~\ref{ex:Burnside} below. 

The possible images of $G$ under ordinary representations are called
\emph{linear groups} over~$\CC$.  More generally, one may study
linear groups over arbitrary fields. See~\cite{MR0335656}
and~\cite{MR1039816} for overviews of results in that direction.
We just note here that for every free abelian group $A$ there exists
at least one field $F$ such that $A$ is a linear group
over~$F$. However, there do exist abelian groups that are not linear
over any field, see~\cite{MR0013164}, cf.~\cite[2.2]{MR0335656}.  Note
also that quotients of linear groups may fail to be linear,
cf.~\cite[Ch.\,6]{MR0335656}. This phenomenon will play a role
in~\ref{realHeis} below. 

In the present notes, we are mainly interested in that part of $G$
that cannot be understood by means of ordinary representations,
namely, the intersection $\KO{G}$ of all kernels of ordinary linear
representations. A detailed overview over the results of the present paper
will be given in~\ref{overview} below; we give some coarse indications
here before we introduce more specific notation. %

We will show that $\KO{G}$ is a central subgroup
of~$G$ if $G$ belongs to the class $\Conn\ProLie$ of all connected
pro-Lie groups (in particular, if $G$ is locally compact and
connected).  Moreover, we investigate the class
$\KO{\mathcal{G}}:=\smallset{\KO{G}}{G\in\mathcal{G}}$ for different
classes $\mathcal{G}$ of groups.  For the class $\Conn\Lie$ of
connected Lie groups, in particular, we show in~\ref{KOconnLie} that
$\KO{\Conn\Lie}$ is the class $\CgAL$ of compactly generated abelian
Lie groups. Thus this class is as large as possible (after the
observation that $\KO{G}$ is central for each $G\in\Conn\Lie$,
cf.~\ref{connCenter}).  The class $\KO{\Conn\ProLie}$ contains all
connected abelian pro-Lie groups and all compact abelian groups,
see~\ref{KOconnProLie}.
For the class $\Conn\LCG$ of connected locally compact groups it turns
out that there is a somewhat surprising bound on the dimension of
members of~$\KO{\Conn\LCG}$, see~\ref{KorConnLCGfinGen}. 

\begin{nota}
  For topological groups $G$, $H$ let $\Hom{G}{H}$ denote the set of
  continuous homomorphisms from $G$ to~$H$. If $G$ and $H$ are
  (topological) vector spaces over~$\FF$, we write $\Hom[\FF]{G}{H}$
  for the set of continuous $\FF$-linear homomorphisms.
  We put
  \[
  \OR{G} := \bigcup\limits_{n\in\NN}\Hom{G}{\GL[n]\CC}, \text{ \ then \ }
  \KO{G} = \bigcap\limits_{\rho\in\OR{G}}\ker\rho \,.
  \]
  Clearly $\KO{G}$ is a closed normal subgroup of~$G$, in fact, it is
  fully invariant (i.e., each endomorphism of the topological
  group~$G$ maps $\KO{G}$ into itself, see~\ref{fullyInv} below).
\end{nota}

In order to keep notation simple, we also consider
continuous homomorphisms from $G$ to the group $\GL{(V)}$ of all linear
bijections of a vector space $V$ of finite dimension~$n$ over
$\FF\in\{\RR,\CC\}$. Note that this does not mean that we consider our
problem in greater generality because $\GL{(V)}$ is isomorphic to
$\GL[n]{\FF}\le\GL[n]{\CC}$.

\begin{rems}
  Our present problem bears some similarity to questions that arise in
  character theory. E.g., for a locally compact abelian group $G$ the
  ordinary representations may be reduced to collections of
  homomorphisms from $G$ into $\GL[1]{\CC}\cong\CC^\times\cong
  \RR\times\TT$ where $\TT:=\RR/\ZZ$ as usual. The elements of
  $\dual{G}:=\Hom{G}{\TT}$ are called \emph{characters} of~$G$ while
  those of $\Hom{G}{\RR}$ are the \emph{real characters}\footnote{ %
    One should not confuse this notion of real character (of
    topological abelian groups) with the usage of the term in the
    theory of characters of finite groups where it denotes a class
    function assuming real values, cf.~\cite[p.\,56\,ff]{MR1050762}.
    Take the direct product $\prod_{p\in\PP}\ZZ(p^\infty)$ with the
    discrete topology. The direct sum
    $\bigoplus_{p\in\PP}\ZZ(p^\infty) \cong \QQ/\ZZ$ is the torsion
    subgroup, but the full product is isomorphic to
    $\QQ^{(2^{\aleph_0})}\oplus\QQ/\ZZ$. The latter is isomorphic to
    $(\RR/\ZZ)_{\mathrm{discr}}$ and to
    $(\RR\times\QQ/\ZZ)_{\mathrm{discr}}$. The projection onto the
    torsionfree summand
    $\QQ^{(2^{\aleph_0})}\cong\RR_{\mathrm{discr}}$ is a real
    character, indeed.}. %
  Since $\Hom{\RR}{\TT}$ separates points we have
  $\bigcap_{\chi\in\dual{G}}\ker\chi \subseteq
  \bigcap_{\rho\in\Hom{G}{\RR}}\ker\rho$.  Pontryagin's duality theory
  for locally compact abelian groups (cf.~\cite[Ch.\,F]{MR2226087})
  uses characters, it rests on the fact that $\Hom{G}{\TT}$ separates
  points for each locally compact abelian group. This is no longer
  true for arbitrary topological abelian groups,
  cf.~\cite[23.32]{MR551496}. The intersection over the kernels of
  real characters has been identified in~\cite{MR0214697},
  cf.~\cite{MR0089361}, \cite{MR2413959}.
\end{rems}

\begin{rema}
  Our present problem disappears if we take a local (or, rather, an
  infinitesimal) point of view. Indeed Ado's Theorem
  (\cite{MR0027753}, see~\cite{MR0030946} for an English translation,
  cf. also~\cite{MR0028829}, \cite{MR0032613}, \cite{MR0201576})
  asserts that every Lie algebra~$\mathfrak{g}$ of finite dimension
  over~$\RR$ or~$\CC$ has a faithful ordinary representation; i.e. a
  faithful homomorphism into $\gL[n]{\RR}$ for some~$n$. Thus a
  single, suitably chosen representation suffices to show that
  $\ko{\mathfrak{g}}$ is trivial, where
  \[
  \ko{\mathfrak{g}} :=
  \bigcap_{n\in\NN}\quad\bigcap_{\rho\in\Hom{\mathfrak{g}}{\gL[n]\CC}}\ker\rho \,.
  \]
  For a pro-Lie algebra~$\mathfrak{g}$ (i.e., a projective limit
  $\mathfrak{g}$ of Lie algebras of finite dimension over~$\RR$ such
  that $\mathfrak{g}$ is complete as a topological vector space,
  cf.~\cite[Ch.\,7]{MR2337107}) one also knows that
  $\ko{\mathfrak{g}}$ is trivial.

  In particular, there is no useful relationship between
  $\KO{G}$ and $\ko{\lie{G}}$ if $G$ is a group which has Lie
  algebra~$\lie{G}$ (in the sense of~\ref{addLie} below).
  However, the Lie algebra will be useful to construct ordinary
  representations of a pro-Lie group, see~\ref{connCenter}. 
\end{rema}

\begin{ndef}[Some classes of groups]
  The following will be of interest to us here: 
  \begin{description}
  \item[$\TG:$] topological Hausdorff groups,
  \item[$\ProLie:$] pro-Lie groups, i.e., complete projective limits
    of Lie groups,
  \item[$\CG:$] compact groups, 
  \item[$\LCG:$] locally compact groups,
  \item[$\LCA:$] locally compact abelian groups,
  \item[$\Lie:$] Lie groups (without separability assumptions, i.e.,
    including \emph{all} discrete groups),
  \item[$\SepLie:$] separable Lie groups (i.e., the $\sigma$-compact
    members of~$\Lie$),
  \item[$\CgAL = \LCA\cap\SepLie:$] compactly generated abelian Lie
    groups, 
  \end{description}
  and the classes $\Ab{\mathcal{G}}$, $\Conn{\mathcal{G}}$ or
  $\almConn{\mathcal{G}}$ consisting of the abelian, connected or
  almost connected members of the class~$\mathcal{G}$,
  respectively. Here a group $G$ is called \emph{almost connected} if
  the quotient $G/G_0$ modulo the connected component $G_0$ is
  compact.
  
  For sentimental historical reasons, we write $\LCA$ and $\CA$
  instead of the more systematic $\Ab{\LCG}$ and $\Ab\CG$,
  respectively.  The diagram
  in Figure~\ref{fig:inclusions} indicates the inclusions between
  these classes.
  \begin{figure}\centering    
    \begin{footnotesize}
      \xymatrix%
      {
        && \TG \\
        & \LCG\ar@{-}[ur]     &             &  \ProLie\ar@{-}[ul] \\
        &&{\CG}\ar@{-}[ul]\ar@{-}[ur]
        &&{\almConn\ProLie}\ar@{-}[ul] \\
        &\LCA\ar@{-}[uu] & \Ab\ProLie\ar@{-}[uur]
        &\Lie\ar@{-}[uu]\ar@{-}@/^{3ex}/[uull]&
        \Conn\ProLie\ar@{-}[u]         \\
        \Conn\LCA\ar@{-}[ur]\ar@{-}[urr] &
        \Ab\Lie\ar@{-}[u]\ar@{-}[ur]\ar@{-}[urr]&
        \Conn{\Ab\ProLie}\ar@{-}[u]\ar@{-}@/_{1ex}/[urr] &
        \SepLie\ar@{-}[u]&
        \Conn\LCG\ar@{-}[u]\ar@/^{5ex}/@{-}[uuulll]\\
        &&& \almConn\Lie\ar@{-}[u]\ar@/^{1ex}/@{-}[uuur] \\
        && \CgAL\ar@{-}[uul]\ar@{-}[uur] &&
        \Conn\Lie\ar@{-}[uu]\ar@{-}[ul] }
      \caption{Inclusions between classes of topological groups}
      \label{fig:inclusions}
    \end{footnotesize}
\end{figure}
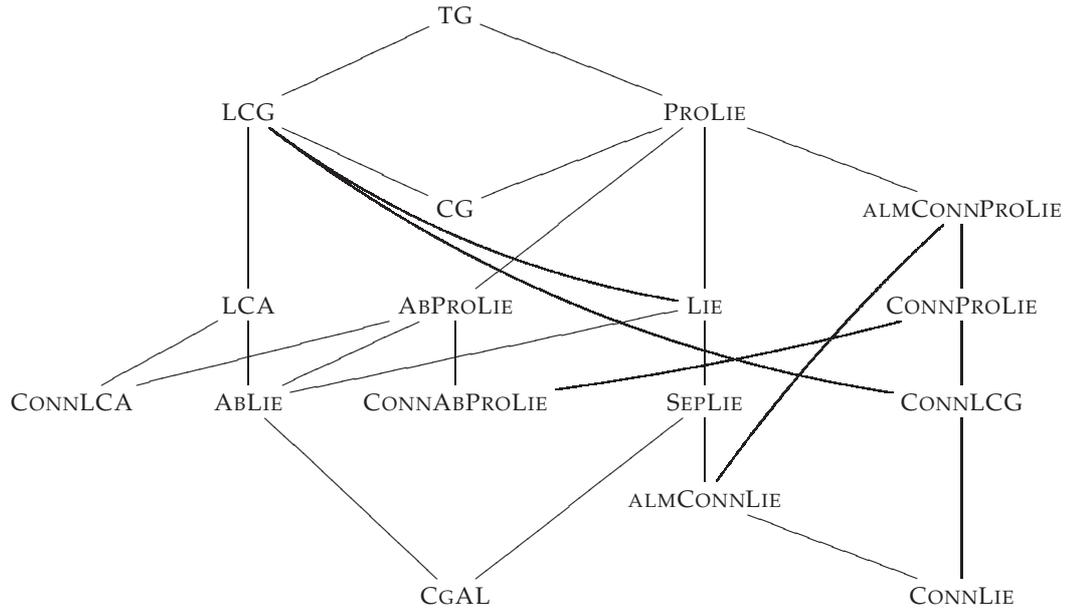
\end{ndef}

Note that we only consider Hausdorff groups; otherwise, the closure of the
trivial subgroup would occur inside $\KO{G}$ throughout.

As it is customary in the theory of locally compact groups, we do not
include separability in the definition of a Lie group. This means that
\emph{every} discrete group is a Lie group, and it secures (via the
solution of Hilbert's Fifth Problem, cf.~\cite{MR0058607},
\cite{MR0073104}) that every locally euclidean group is a Lie
group. Thus the additive group $\RR_{\mathrm{discr}}$ of real numbers
with the discrete topology belongs to ${\Lie\setminus\SepLie}$ (it is
not a member of $\CgAL$), and the identity from $\RR_{\mathrm{discr}}$
to~$\RR$ is a bijective morphism of Lie groups which is not open.  If
one wants to use the Open Mapping Theorem (which is indeed one of the
major reasons to require separability) one has to be careful and make
sure that the domain of the mapping is $\sigma$-compact.  Note that
every closed subgroup of an \emph{almost connected} Lie group belongs
to~$\SepLie$.  In our present context separability does not appear to
be of much help
(see~\ref{ex:separabilityDoesNotHelp}\ref{separabilityDoesNotHelp},
\ref{KOsepLie}) while almost connectedness is a useful condition for
Lie groups (cf.~\ref{almConnLie}) where it actually means that the
number of connected components is finite.

\begin{ndef}[Overview of results]\label{overview}
  In the present paper, we obtain the following. 
  \begin{itemize}
    \item $\KOf$ is a functor that preserves products,
      see~\ref{fullyInv}, \ref{cartProd}. 
    \item For each $G\in\Conn\ProLie$ we have
      $\KO{G}\subseteq\Z{G}\cap\closure{G'}$, see~\ref{connCenter}
      and~\ref{KOinZcapComm}.
  \item $\KO{\almConn\Lie} = \KO{\Conn\Lie} = \CgAL$,
    see~\ref{KOconnLie} and~\ref{almConnLie}.
  \item $\KO{G}$ is trivial for every compact and every abelian
    proto-Lie group $G$, see~\ref{KOCG}.
  \item For $G\in\Conn\LCG$ the connected component $(\KO{G})_0$ of
    $\KO{G}$ has a finitely generated dense subgroup. Thus the weight
    of $(\KO{G})_0$ is bounded by~$2^{\aleph_0}$,
    see~\ref{connKOsmall}.
  \item If\/ $G\in\Conn\LCG$ is solvable then $\closure{G'}$ has a
    finitely generated dense subgroup, cf.~\ref{sKOfiniteIfSolv}. This
    implies that the weight of $\KO{G}$ is bounded by~$2^{\aleph_0}$.
  \item The inclusions 
    $\Conn{\Ab\ProLie} \subset \aP(\CgAL \cup\CA) \subseteq 
    \KO{\Conn\ProLie} \subseteq \cS(\Conn{\Ab\ProLie})$
    are established in~\ref{KOconnProLie} and~\ref{prob:KOConnProLie}.
  \item $\KO{\Conn\LCG}$ contains those $A\in\LCA$ that possess a
    finitely generated dense subgroup, see~\ref{KorConnLCGfinGen}.
  \item For each $G\in\Conn\LCG$ there exist $A\in\CA$ and natural
    numbers~$e,f$ such that the connected component $A_0$ is monothetic and
    $\KO{G}\cong\ZZ^f\times A\times\RR^e$, see~\ref{KorConnLCGfinGen}.
    In particular, the dimension of members of\/~$\KO{\Conn\LCG}$ is
    bounded by~$2^{\aleph_0}$.
  \end{itemize}

\bigbreak\noindent
  The latter two of these results mean that $\KO{\Conn\LCG}$ is
  sandwiched between the class $\Small\LCA$ of groups of the form
  $\ZZ^f\times F\times\RR^e$ where~$F$ is a compact group with a
  finitely generated dense subgroup and the class of groups of the
  form $\ZZ^f\times A\times\RR^e$ where $A$ is compact and its
  connected component $A_0$ has a finitely generated dense subgroup;
  here~$f$ and~$e$ may be arbitrary natural numbers.
\enlargethispage{3mm}
  
  The classes $\KO{\SepLie}$ and $\KO{\LCG}$ are large and
  complicated, see \ref{KOsepLie}, \ref{KOLCG}
  and~\ref{KOLCGnotClosedUnderQ}. Some open problems are stated in
  Section~\ref{sec:openQuestions}.  Relevant results about (non-)
  linear(ity of) groups are collected in Section~\ref{sec:appendix}.
\end{ndef}

\goodbreak
\section{Basic results}

In this section we discuss functoriality of $\KOf$ and the behavior
of this functor with respect to homomorphisms (in particular
quotients) and products.

\begin{lemm}\label{fullyInv}
  For $\varphi\in\Hom{G}{H}$ we have $\varphi(\KO{G}) \le \KO{H}$.

  In particular, mapping $G$ to $\KO{G}$ and $\varphi$ to
  $\KO{\varphi} := \varphi|_{\KO{G}}^{\KO{H}}$
  implements a functor from the category of topological groups to
  itself.
\end{lemm}
\begin{proof}
  For each $\rho\in\OR{H}$ we have $\rho\circ\varphi\in\OR{G}$, and
  $\rho(\varphi(x))=\id$ follows for each $x\in\KO{G}$.
\end{proof}

\begin{prop}\label{KOquotients}
  Let\/ $G$ be a group, and let\/ $N$ be a subgroup of~$G$.
  \begin{enumerate}
  \item In any case the group $\KO{N}$ is contained in~$\KO{G}$.
  \item If\/ $N$ is a normal subgroup of\/ $G$ then $\KO{G}N/N$ is
    contained in $\KO{G/N}$.
  \item\label{wellBehavedQuotient}
    If\/ $N$ is normal and contained in~$\KO{G}$ then $\KO{G}/N =
    \KO{G/N}$.
  \item\label{radical}
    The subgroup $\KO{G}$ is a radical in the sense that $\KO{G/\KO{G}}$
    is trivial.
  \end{enumerate}
\end{prop}
\begin{proof}
  The first two assertions follow from~\ref{fullyInv} using the
  inclusion map $\iota\colon N\to G$ and the canonical quotient map
  $q_N\in\Hom{G}{G/N}$.

  Now assume that $N$ is normal in $G$ and contained in~$\KO{G}$.  For
  $x\in G\setminus\KO{G}$ there exists $\rho\in\OR{G}$ with
  $\rho(x)\ne\id$, and $\rho$ factors as $\rho=\lambda\circ q_N$
  because $N\le\KO{G}\le\ker\rho$. Thus there exists
  $\lambda\in\OR{G/N}$ with $\lambda(xN)\ne\id$. This means $\KO{G}/N
  \ge \KO{G/N}$; the inclusion $\KO{G}/N \le \KO{G/N}$ is clear
  already. Thus~\ref{wellBehavedQuotient} is established, and the last
  assertion follows.
\end{proof}

From~\ref{KOquotients}\ref{radical} we see that $\KO{G}$ is a sort of
``radical'' of the group~$G$. 
Note that $\KO{G/N}$ may be much larger than $\KO{G}N/N$ if $N$ is not
contained in $\KO{G}$; the example in~\ref{realHeis} is instructive
here, again.

\goodbreak %
The category $\TG$ and its full subcategories $\ProLie$,
$\Conn\ProLie$, $\Ab\ProLie$ and $\Conn{\Ab\ProLie}$ are closed under
arbitrary products, and these are as expected (i.e., cartesian
products with the product topology). A category of topological groups
may contain products (in the categorical sense) that are endowed with
a topology that is different from the product topology; e.g., this
happens in the categories $\LCG$ and~$\LCA$
(cf.~\cite[16.22]{MR2226087}).  However, products in $\Conn\LCG$ are
the same as those in $\TG$ (in particular, they exist only if all but
a finite number of the factors are compact),
see~\cite[16.23]{MR2226087}.

\begin{defs}
  For $\mathcal{G} \subseteq \TG$ let $\fP(\mathcal{G})$ denote the
  class of all cartesian products of \emph{finitely} many members
  of~$\mathcal{G}$. By $\aP(\mathcal{G})$ we mean the class of
  arbitrary cartesian products of members of~$\mathcal{G}$. In any
  case, we use the product topology.

  The class $\aS(\mathcal{G})$ consists of all subgroups of members
  of\/~$\mathcal{G}$ while $\cS(\mathcal{G})$ contains only the closed
  subgroups (which is more reasonable if one studies classes of
  complete groups as we do here). 

  By $\hQ(\mathcal{G})$ we denote the class of all Hausdorff
  quotient groups of members of~$\mathcal{G}$, i.e. the class of all
  groups $G/N$ where $G\in\mathcal{G}$ and $N$ is a \emph{closed}
  normal subgroup of~$G$.
  Finally, let $\cQ(\mathcal{G})$ be the class of all (Hausdorff)
  completions of members of~$\mathcal{G}$. 
\end{defs}

\begin{rems}
  A topological group need not have a
  completion, cf.~\cite[III\,\S\,3$\cdot$4, Thm.\,1]{MR979294}.  %
  The classes $\ProLie$ and $\Conn\ProLie$
  are not closed under~$\hQ$,
  see~\cite[4.11]{MR2337107}. However, Hausdorff quotients of pro-Lie
  groups are proto-Lie groups (see~\cite[4.1]{MR2337107}), i.e., they
  possess completions which belong to $\ProLie$, again. Thus
  $\cQ(\ProLie)=\ProLie\subsetneqq\hQ(\ProLie)$.  Locally compact
  groups are complete anyway (see~\cite[III\,\S3$\cdot$3,
  Cor.\,1]{MR979294}, cf.~\cite[1.31]{MR2337107}
  or~\cite[8.25]{MR2226087}), and $\hQ(\LCG)=\LCG=\cQ(\LCG)$.
\end{rems}

The following corollary to~\ref{KOquotients} is applicable to the
class $\Conn\LCG$ and its subclass $\Conn\Lie$, cf.~\ref{connCenter}.
The class $\Conn\ProLie$ is problematic because it is not closed
under~$\hQ$.
Some restriction of the sort ``$\KO{G}\le\Z{G}$'' is necessary,
cf.~\ref{KOLCGnotClosedUnderQ}. 

\begin{coro}\label{quotients}
  Consider\/ $\mathcal{G}\subseteq\TG$ with
  $\hQ(\mathcal{G})=\mathcal{G}$ and assume that\/ $\KO{G}\le\Z{G}$
  holds for each $G\in\mathcal{G}$. Then
  $\KO{\mathcal{G}}=\hQ(\KO{\mathcal{G}})$.
\qed
\end{coro}

\begin{lemm}[{\cite[11.18, 13.13]{MR644485},
    cf.~\cite[4.28]{MR2337107}}]%
  \label{proLieModLCGisComplete}
  Let $N$ be a normal subgroup of a pro-Lie group~$G$. If $N$ is
  locally compact then $G/N$ is
  complete, and thus a pro-Lie group.  
\qed
\end{lemm}

\begin{prop}\label{cartProd}
  The functor $\KOf$ preserves products in the category~$\TG$, in
  fact, we have $\KO{\prod_{j\in J}A_j} = \prod_{j\in J}\KO{A_j}$. For
  $\mathcal{G}\subseteq\TG$ this means\/
  $\fP(\KO{\mathcal{G}})=\KO{\mathcal{G}}$ if\/
  $\fP(\mathcal{G})=\mathcal{G}$ and
  $\aP(\KO{\mathcal{G}})=\KO{\mathcal{G}}$ if\/
  $\aP(\mathcal{G})=\mathcal{G}$.
\end{prop}
\begin{proof}
  Let $\prod_{j\in J}A_j$ be a cartesian product; the indexing
  set $J$ may be infinite. %
  For $m\in J$ let $\eta_m\colon A_m\to\prod_{j\in J}A_j$ be the
  natural inclusion, and let $\pi_m\colon\prod_{j\in J}A_j\to A_m$ be
  the natural projection.

  If $\rho\colon\prod_{j\in J}A_j\to\GL[n]{\CC}$ is an ordinary
  representation of the product then $\rho\circ\eta_m$ is an ordinary
  representation of~$A_m$, and we obtain that the subgroup generated
  by $\bigcup_{j\in J}\eta_j(\KO{A_j})$ is contained in
  $\KO{\prod_{j\in J}A_j}$. The
  product $\prod_{j\in J}\KO{A_j}$ is the closure of that subgroup
  and thus also contained in $\KO{\prod_{j\in J}A_j}$.

  Conversely, consider %
  $x\in\prod_{j\in J}A_j\setminus\prod_{j\in J}\KO{A_j}$. Then there
  exists $m\in J$ such that $\pi_m(x)\notin\KO{A_m}$ and we
  find an ordinary representation $\rho_m$ of $A_m$ with
  $\id\ne\rho_m(\pi_m(x)) = (\rho_m\circ\pi_m)(x)$. Since
  $\rho_m\circ\pi_m$ is an ordinary representation of %
  $\prod_{j\in J}A_j$ this shows $x\notin\prod_{j\in J}A_j \setminus
  \KO{\prod_{j\in J}A_j}$. 
\end{proof}

%%%%%%%%%%%%%%%%%%%%%%%%%%%%%%%%%%%%%%%%%%%%%%%%%%%%%%%%%%%%%%%%%%%%%%%
\enlargethispage{5mm}
\goodbreak

\begin{lemm}\label{CpCommSolv}
  Let $G$ be a solvable connected (not necessarily closed) subgroup of\/
  $\GL[n]{\CC}$. Then the following holds.
  \begin{enumerate}
  \item\label{invFlag} There exists a sequence $V_0,\dots,V_n$ of\/
    $G$-invariant subspaces such that for all $j<n$ we have $V_j\le
    V_{j+1}$ and\/ $\dim{V_j}=j$.
  \item For any sequence as in~\ref{invFlag} the commutator group $G'$
    acts trivially on each 
    $V_{j+1}/V_j$.
  \item There are no compact (in particular, no finite) subgroups in
    $G'$ except the trivial one. 
  \end{enumerate}
\end{lemm}
\begin{proof}
  Replacing~$G$ by its closure in~$\GL[n]{\CC}$ we lose neither solvability nor
  connectedness, cf.~\cite[2.9, 7.5]{MR2226087}.  For closed connected
  solvable subgroups of~$\GL[n]\CC$ assertion~\ref{invFlag} is Lie's
  Theorem, cf.~\cite[Thm.\,2.2, Ch.\,III]{MR514561}.

  Since $G$ acts as a subgroup of the abelian group $\GL{(V_{j+1}/V_j)}
  \cong \GL[1]\CC \cong \CC^\times$ on $V_{j+1}/V_j$ the commutator
  group $G'$ acts trivially on that quotient.

  Finally, let $C$ be a compact subgroup of~$G'$. We proceed by
  induction to show that $C$ acts trivially on~$V_j$ for each $j\le
  n$. Indeed, if $C$ acts trivially on $V_j$ it acts on $V_{j+1}$ as a
  subgroup of the group $N_j$ consisting of all $a\in\GL{(V_{j+1})}$
  acting trivially both on $V_j$ and on $V_{j+1}/V_j$. Now $N_j$ is
  isomorphic to the additive group $\Hom[\CC]{V_{j+1}/V_j}{V_j}$. This
  is the additive group of a vector space of finite dimension
  over~$\RR$, and does not contain compact subgroups apart from the
  trivial one.
\end{proof}

\begin{coro}\label{CpCommSolvCor}
  Let $G$ be a solvable connected group. Then every compact subgroup
  of~$G'$ is contained in $\KO{G}$.  
\qed
\end{coro}

Connectedness is a crucial assumption in~\ref{CpCommSolvCor}, as
finite groups show.  Applications of~\ref{CpCommSolvCor} are given
in~\ref{realHeis} and~\ref{exa:HeisenbergWithCpCenter} below.
See also~\ref{maltsev}\ref{maltsev3} and~\ref{nahlus}.

\section{Lie algebras and pro-Lie groups}

\enlargethispage{5mm}
For a Lie group $L$ one model for the Lie algebra is the space
$\Hom{\RR}{L}$ of all one-parameter subgroups, cf.~\cite{MR0409722}.
This point of view works for quite general classes of topological
groups, see~\cite[Ch.\,2]{MR2337107}.

\begin{defs}\label{addLie}
  For a topological group $G$ let $\lie{G}$ denote the space
  $\Hom{\RR}{G}$ endowed with the compact-open topology (i.e., the
  topology of uniform convergence on compact sets).
  We call $\exp\colon\lie{G}\to G\colon X\mapsto X(1)$ the
  \emph{exponential map} for~$G$.
  Multiplication of
  $X\in\lie{G}$ by a scalar $r\in\RR$ is given as $r\,X(t) := X(tr)$.

  Addition and the Lie bracket are more involved, and not defined for
  arbitrary topological groups.
  We say that $G$ \emph{has a Lie algebra} if the following conditions
  are satisfied: 
  \begin{enumerate}
  \item\label{defAdd}
    For all $X,Y\in\lie{G}$ there are elements $X+Y$ and $[X,Y]$
    in~$\lie{G}$ such that 
    \[
    \begin{array}{rcl}
      (X+Y)(t) & = & 
      \lim\limits_{n\to\infty}
      \left( X(\frac{t}{n}) \,Y(\frac{t}{n}) \right)^n
      \text{ \ and}
      \\{}
      [X,Y](t^2)& = &
      \lim\limits_{n\to\infty}
      \comm{X(\frac{t}{n})}{Y(\frac{t}{n})}^{n^2} 
    \end{array}
    \]
    hold for all $t\in\RR$; here $\comm{g}{h}:=ghg^{-1}h^{-1}$ is the
    group commutator.
  \goodbreak
  \item The set $\lie{G}$ is a topological Lie algebra with respect to these
    operations.
  \goodbreak
  \end{enumerate}
  We say that $G$ has a \emph{generating Lie algebra} if it has a Lie
  algebra and the range of the exponential map generates a dense
  subgroup of the connected component~$G_0$.
\end{defs}

\begin{exas}
  In~\cite[3.5]{MR2337107} it is shown that every projective limit $G$
  of Lie groups has a Lie algebra and, moreover, this Lie algebra
  $\lie{G}$ is a pro-Lie algebra: i.e., the filter basis of closed
  ideals of finite co-dimension in $\lie{G}$ converges to~$0$ and
  $\lie{G}$ is complete as a topological vector space.  In particular,
  this Lie algebra is \emph{residually finite-dimensional}; the
  homomorphisms to finite-dimensional Lie algebras separate the
  points.

  Every almost connected pro-Lie group, and thus every almost
  connected locally compact group, has a generating Lie algebra,
  cf.~\cite[4.22]{MR2337107}.
\end{exas}

Our technical machinery culminates in the adjoint representation: 

\goodbreak
\begin{lemm}[{\cite[2.27, 2.28, 2.30, 4.22, 8.1]{MR2337107}}]\label{connProLieSuff}
  Let $G$ be a topological group.
  \begin{enumerate}
  \item For each $g\in G$ there is a unique bijection $\Ad(g)$ of $\lie{G}$
    onto itself such that $gX(t)g^{-1} = \Ad(g)(X)(t)$ holds for all
    $t\in\RR$.
  \item The action $G\times\lie{G}\to\lie{G}\colon (g,X)\mapsto
    \Ad(g)(X)$ is continuous.
  \SaveEnumi
  \end{enumerate}
  \goodbreak
  Now assume that $G$ has a Lie algebra. 
  \begin{enumerate}
  \RecallEnumi
  \item For each $g\in G$ the bijection $\Ad(g)$ is an automorphism of
    the Lie algebra~$\lie{G}$.
  \item The adjoint representation $\Ad\colon G\to\Aut{\lie{G}}$ is a
    continuous linear representation, where $\Aut{\lie{G}}$ is endowed
    with the strong operator topology (i.e., the topology of pointwise
    convergence).
  \item The kernel of $\Ad$ is the centralizer of (the closure of) the
    subgroup generated by the range of the exponential function. 
  \item If $G$ is a pro-Lie group then each ideal of the Lie algebra
    $\lie{G}$ is invariant under the adjoint action of the connected
    component of\/~$G$.
  \end{enumerate}
\end{lemm}

\begin{theo}\label{connCenter}
  If\/ $G$ is a connected pro-Lie group then $\KO{G}$ is contained
  in the center $\Z{G}$ of\/~$G$ and is therefore abelian.
\end{theo}
\begin{proof}
  According to~\ref{connProLieSuff} the adjoint representation $\Ad$
  induces ordinary representations on the finite-dimensional quotients
  of~$\lie{G}$ that separate the points modulo $\ker\Ad$, which equals
  the center of~$G$.
\end{proof}

\begin{coro}
  For every connected locally compact group $G$ the group $\KO{G}$
  is contained in the center $\Z{G}$. 
\qed
\end{coro}

Using~\ref{KOquotients} and~\ref{proLieModLCGisComplete} we infer: 
\begin{coro}\label{quotientsKOConnProLie}
  The class $\KO{\Conn\ProLie}$ is closed under quotients modulo
  locally compact groups while $\KO{\Conn\Lie}$
  and $\KO{\Conn\LCG}$ are closed under arbitrary (Hausdorff)
  quotients.  
\qed
\end{coro}

\begin{rems}\label{ex:separabilityDoesNotHelp}
  \begin{enumerate}
  \item\label{separabilityDoesNotHelp} 
    Theorem~\ref{connCenter} does not extend to the case of arbitrary
    disconnected Lie groups even if we assume separability. In fact
    there are countable discrete groups $G$ with $\KO{G}=G$,
    see~\ref{ex:Burnside} below. However, everything is fine for
    almost connected Lie groups, see~\ref{almConnLie}.
  \item There are groups in $\LCG\setminus\ProLie$ such that $\Ad$ is
    a faithful representation (of infinite degree) but every ordinary
    representation is trivial on the connected component,
    cf.~\ref{powerD}.
  \end{enumerate}
\end{rems}

The following observation will be useful to obtain restrictions on the
structure and size (measured by the dimension, i.e., the rank of the
Pontryagin dual) of members of $\KO{\Conn\LCG}$;
cf.~\ref{sKOfiniteIfSolv}, \ref{connKOsmall} and~\ref{sKOfinite} below. %
The passage to the
\emph{closure} of the commutator group is essential,
cf.~\ref{SL2timesH} and~\ref{exa:monotheticConnCA}.

\begin{lemm}\label{KOinZcapComm}
  \begin{enumerate}
  \item If $G\in\ProLie$ then $\KO{G}\le\closure{G'}$.
  \item If $G\in\Conn\ProLie$ (in particular, if $G\in\Conn\LCG$) then
    $\KO{G} \le \Z{G}\cap\closure{G'}$.
  \end{enumerate}
\end{lemm}
\begin{proof}
  The quotient $G/\closure{G'}$ is an abelian proto-Lie group
  (see~\cite[4.1]{MR2337107}). Thus $\OR{G/\closure{G'}}$ separates
  the points, and $\KO{G/\closure{G'}}$ is trivial. For each $x\in
  G\setminus\closure{G'}$ we thus find some
  $\rho\in\OR{G/\closure{G'}}$ with $x\,\closure{G'} \notin\ker\rho$,
  and composing $\rho$ with the quotient map we find $\rho'\in\OR{G}$
  such that $x\notin\ker{\rho'}$.

  If $G$ is also connected then $\KO{G}\le\Z{G}$ has been
  established~\ref{connCenter}.
\end{proof}

\section{Almost connected Lie groups}

\begin{exam}\label{realHeis}
  The following example has been around for quite some time,
  see~\cite[4.14, p.\,191]{MR0073104} or~\cite{MR0327979}.
  We use it to show that $\RR/\ZZ$ belongs to $\KO{\Conn\Lie}$. 
  Let $H$ be the real Heisenberg group; i.e., $H=\RR^3$ with the
  multiplication $(a,b,x)*(c,d,y)=(a+c,b+d,x+y+ad-bc)$. Clearly this
  is a connected Lie group, and the center $\{0\}^2\times\RR$
  coincides with the commutator group~$H'$.

  The cyclic subgroup $Z := \{0\}^2\times\ZZ$ is closed and central
  in~$H$. Thus $H/Z$ is a connected nilpotent (and thus solvable) Lie
  group. According to~\ref{CpCommSolvCor} the compact central subgroup
  $C:=(\{0\}^2\times\RR)/Z \cong \RR/\ZZ$ of $(H/Z)'$ is contained
  in~$\KO{H/Z}$. On the other hand, the quotient $H/C\cong\RR^2$ has
  a faithful ordinary representation.
  Therefore, we obtain $\KO{H/Z} = C \cong \RR/\ZZ$. 
\end{exam}

For every connected Lie group $L$ with semi-simple Lie algebra one
knows that $\KO{L}$ is a discrete (and thus central) normal subgroup,
cf.~\cite[5.3.6 Thm.\,8, p.\,264]{MR1064110}. In fact, for any such
group one can read off $\KO{L}$ from~\cite[Table\,10,
p.\,318\,f]{MR1064110}. The relevant information is accessible
algorithmically, see~\cite{RealLIE},~\cite{MR1270178}. 
For our present purposes, we require explicit
knowledge of the case where $L$ is the simply connected covering
of~$\PSL[2]{\RR}$. 

\begin{exam}\label{Sl2}
  Let $S$ denote the simply connected covering of $\PSL[2]\RR$. The
  center $\Z{S}$ of $S$ is infinite cyclic, but no
  proper covering of $\SL[2]\RR$ admits a faithful ordinary
  representation; see~\cite[Table\,10, p.\,318\,f]{MR1064110} (where
  the Lie algebra $\sL[2]\RR$ occurs in the guise of $\sP[2]{\RR} =
  \sP[4p+2]{\RR}$ for~$p=0$), cf. also~\cite{MR0327979} and~\cite[95.9,
  95.10]{MR1384300}. Thus we obtain $\KO{S} =
  \smallset{z^2}{z\in\Z{S}} \cong \ZZ$.

  Passing from $S$ to the quotient $S_d := S/\smallset{z^{2d}}{z\in\Z{S}}$
  we find $\KO{S_d} \cong \ZZ/d\,\ZZ$ for any nonnegative
  integer~$d$.
\end{exam}

Instead of the simple group $\PSL[2]\RR$ we could use any other simple
Lie group with infinite fundamental group
(cf.~\ref{simpleLieInfinitePi}) for the construction in~\ref{Sl2}.

\begin{exam}\label{SL2timesH}
  Again, let $S$ denote the simply connected covering of $\SL[2]\RR$,
  and let $H$ be the real Heisenberg group (cf.~\ref{realHeis}).  Pick
  a generator $\zeta$ for the center of~$S$. Then the subgroup $K:=
  \smallset{(\zeta^{-2z},(0,0,z))}{z\in\ZZ}$ is closed and
  central. Passing to the quotient $G:=({S\times H})/K$ amounts to an
  identification of $\zeta^2$ with $(0,0,1)$.

  Composing the inclusion maps of the two factors $S$ and $H$ with the
  quotient map modulo $K$ we obtain continuous homomorphisms $\eta_S$
  and $\eta_H$ from $S$ and $H$, respectively, into~$G$. Composition of
  ordinary representations of $G$ with $\eta_X$ then yields ordinary
  representations of~$X\in\{S,H\}$. 

  Any ordinary representation of $S$ has $\zeta^2$ in its kernel
  (cf.~\ref{Sl2}). Therefore, any ordinary representation $\varphi$ of
  $G$ yields a representation $\varphi\circ\eta_H$ of $H$ with
  $(0,0,1)$ in its kernel. According to~\ref{realHeis} this implies 
  $\{0\}^2\times\RR \le \ker(\varphi\circ\eta_H)$.  Therefore, the
  subgroup 
  \[
  R:=\smallset{(\zeta^{2z},(0,0,x))}{z\in\ZZ,x\in\RR} / K
  \]
  is contained in~$\KO{G}$. Since $G/R\cong\SL[2]\RR\times\RR^2$ has a
  faithful ordinary representation we obtain $\KO{G}=R\cong\RR$. 
\end{exam}

The examples collected so far suffice to determine the class
$\KO{\Conn\Lie}$. For the proof of~\ref{KOconnLie} we need the
following explicit description of the class $\CgAL$ (see
also~\cite{MR2542208}).

\begin{prop}\label{membersCgAL}
  The elements of\/ $\CgAL$ are precisely those of the form 
  \[
  \ZZ^f\times \prod_{j\in J}(\ZZ/d_j\ZZ) \times (\RR/\ZZ)^c
  \times\RR^e
  \eqno{(*)}
  \]
  where $f,e,c$ are nonnegative integers and $(d_j)_{j\in J}$ is a
  finite family of positive integers. 
\end{prop}
\begin{proof}
  The groups with a product decomposition as in~$(*)$ are clearly
  compactly generated Lie groups. Among all locally compact abelian
  groups, those of the form~$(*)$ are characterized by the property of
  being compactly generated and having no small subgroups,
  cf.~\cite[21.17]{MR2226087}. Since a Lie group has no small
  subgroups, the assertion follows. 
\end{proof}

\begin{theo}\label{KOconnLie}
  Exactly the compactly generated abelian Lie groups occur as $\KO{L}$
  for a suitable connected Lie group~$L$. In other words, we have
  $\KO{\Conn\Lie}=\CgAL$.
\end{theo}
\begin{proof}
  Each member of the class $\CgAL$ is isomorphic to a product
  \[
  \ZZ^f\times \prod_{j\in J}(\ZZ/d_j\ZZ) \times (\RR/\ZZ)^c
  \times\RR^e
  \]
  where $f,e,c$ are nonnegative integers and $(d_j)_{j\in J}$ is a
  finite family of positive integers, cf.~\ref{membersCgAL}. Thus the
  assertion of our theorem follows from~\ref{cartProd} together with
  the examples given in~\ref{realHeis}, \ref{Sl2},
  and~\ref{SL2timesH}. Explicitly, the group
  \[
  L:= S^f \times \prod_{j\in J} S_{d_j} \times (H/Z)^c
  \times ((S\times H)/K)^e
  \]
  satisfies our requirements. In order to prove the converse it
  suffices to note that $\KO{G}$ is a closed abelian subgroup of $G$
  if $G\in\Conn\Lie$, and thus lies in~$\CgAL$. 
\end{proof}

\begin{rema}
  Even for a simply connected Lie group~$L$ with simple complex Lie
  algebra it need not be true that $L/\KO{L}$ has an
  \emph{irreducible} faithful ordinary representation. For instance,
  consider the simply connected covering $\Spin{8}{\CC}$ of
  $\SO{8}{\CC}$; the center of that group is non-cyclic and cannot be
  mapped faithfully into the centralizer of an irreducible ordinary
  representation because that centralizer is a subgroup of the
  multiplicative group of Hamilton's quaternions by Schur's
  Lemma. However, the group $\Spin{8}{\CC}$ is linear, and
  $\KO{\Spin{8}{\CC}}$ is trivial. 
\end{rema}

\begin{theo}\label{finiteIndex}
  If\/ $U$ is an open normal subgroup of finite index in $G$ then
  $\KO{U} = \KO{G}$. 
\end{theo}
\begin{proof}
  For any $\rho\in\OR{U}$ the $U$-module $M_\rho$ associated with
  $\rho$ yields an induced module $L_\rho:=\ind{U}{G}M_\rho$,
  cf.~\cite[8.4, p.\,230\,f]{MR1357169}
  or~\cite[XVIII~7.3]{MR783636}. The corresponding representation
  $\lambda\colon G\to\GL{(L_\rho)}$ can now be combined with the
  faithful regular representation $\mu\colon G/U\to\GL{(\CC^{G/U})}$
  of the finite quotient to obtain an ordinary representation of $G$
  whose kernel is contained in $\ker\rho$.  Since $\rho\in\OR{U}$ was
  arbitrary we obtain $\KO{G}\le\KO{U}$. The reverse inclusion is
  clear from~\ref{KOquotients}.
\end{proof}

  In an almost connected Lie group~$G$ the discrete and compact
  quotient $G/G_0$ is finite. Thus~\ref{finiteIndex} yields: 
\begin{coro}\label{almConnLie}
  If\/ $G$ is an almost connected Lie group then
  $\KO{G}=\KO{G_0}$. In particular, we have $\KO{\almConn\Lie} =
  \KO{\Conn\Lie} = \CgAL$. 
\qed
\end{coro}

\begin{rema}
  One would of course like to extend~\ref{almConnLie} to the classes
  $\almConn\LCG \subset \almConn\ProLie$. %
  The ordinary representations of the compact quotient separate the
  points; thus $\KO{G}\le G_0$ if $G\in\almConn\ProLie$. If one wants
  to proceed along the lines of the proof of~\ref{finiteIndex} then
  there remains the problem to extend a representation of $G_0$ via
  induction. This question is treated by Mackey~\cite{MR0042420} where
  an invariant scalar product is assumed and the group in question is
  required to be locally compact \emph{and separable}.
  Note that~\ref{finiteIndex} yields $\KO{U}=\KO{G}$ for each open
  normal subgroup of~$G$ but $\KO{G_0}$ might still be smaller
  although $G_0$ is the intersection of those open normal subgroups
  (cf.~\cite[6.8]{MR2226087}).
\end{rema}

\section{Examples that are not Lie groups}

A natural source of locally compact groups that are not Lie groups
originates from the fact that the class $\CG$ is closed under
arbitrary cartesian products. Another well-understood (and rich)
class of locally compact groups is the class $\LCA$.
Both classes do not contribute to $\KO{\LCG}$: 

\begin{theo}\label{KOCG}\label{KOLCA}
  The class $\KO{\CG}\cup\KO{\LCA}\cup\KO{\Ab\ProLie}$ consists of the
  trivial group alone. In fact $\KO{G}$ is trivial for every abelian
  proto-Lie group $G$.
\end{theo}
\begin{proof}
  The Peter-Weyl Theorem asserts that the
  ordinary representations separate the points in any compact group,
  cf.~\cite[14.33]{MR2226087}. Thus $\KO{\CG}$ contains only the
  trivial group. 

  It remains to consider $A\in\LCA$. Every character of~$A$ is a continuous
  homomorphism from $A$ into $\RR/\ZZ\cong\SU[1]{\CC}<\GL[1]{\CC}$ and
  thus belongs to~$\OR{A}$. Pontryagin duality
  (cf.~\cite[22.6]{MR2226087}) implies that the characters separate
  the points of~$A$. Thus $\KO{A}$ is trivial.

  For a proto-Lie group $G$ the continuous homomorphisms to Lie groups
  separate the points. If $G$ is abelian then it suffices to consider
  homomorphisms from $G$ to abelian Lie groups, and it follows that
  $\KO{G}$ is trivial.  
\end{proof}

\begin{exam}\label{powerD}
  Let $C$ be a compact simple %
  non-abelian group and let $D$ be any infinite discrete group.  For
  instance, the group $C=\SO{3}{\RR}$ would do --- in any case, $C$
  will be connected (see~\cite[4.13]{MR2226087}).  The product
  topology turns $C^D$ into a compact connected group belonging to
  $\Conn\LCG\subset\Conn\ProLie$ but not to~$\Lie$.  We form the
  semidirect product $G:=D\ltimes C^D$ where conjugation with $d\in D$
  maps $(c_j)_{j\in D} \in C^D$ to $(b_j)_{j\in D}$ with
  $b_j=c_{dj}$. Then $G\in\LCG$.

  The closed normal subgroups of~$C^D$ are in one-to-one
  correspondence with the lattice of subsets of~$D$,
  cf.~\cite[3.12]{MR1859180}: to $J\subseteq D$ we associate the 
  product $\prod_{j\in D}B_j$ where $B_j=C_j$ if $j\in J$ and $B_j$
  is trivial if $j\in D \setminus J$. 
  Consequently, every non-trivial closed normal subgroup of $G$
  contains the connected component $G_0\cong C^D$, and we
  have $\KO{G}=G_0\cong C^D$ here.

  Our investigation of the normal subgroups also makes clear that $G$
  does not belong to~$\ProLie$.
\end{exam}

\begin{exam}
  Let $S$ again denote the universal covering group of
  $\SL[2]{\RR}$. As in~\cite[Ex.\,0.6]{MR1391958} we consider the
  filter basis $\mathcal{N}(S)^\times$ of all \emph{nontrivial} subgroups of
  the center of~$S$.  The projective limit
  $G:=\lim_{N\in\mathcal{N}(S)^\times}S/N$ is a connected locally
  compact group (see~\cite[2.12]{MR1391958}) with a center $\Z{G}$
  isomorphic to the universal zero dimensional compactification
  of~$\ZZ$.

  In other words $\Z{G}$ is isomorphic to $\prod_{p\in\PP}\ZZ_p$ where $\PP$
  is the set of primes and $\ZZ_p$ is the additive group of $p$-adic
  integers. Note that $\Z{G}$ has a unique quotient of order~$2$
  because $\Z{G}/\smallset{z^2}{z\in\Z{G}} \cong \prod_{p\in\PP}\ZZ_p /
  \prod_{p\in\PP}2\,\ZZ_p \cong \ZZ_2/2\,\ZZ_2 \cong \ZZ/2\,\ZZ$.
  From~\cite[2.14]{MR1391958} we infer that $G$ and
  $G/\Z{G}\cong\PSL[2]{\RR}$ have essentially the same Lie algebra. 

  As $G/\Z{G}\cong\PSL[2]{\RR}$ is a simple group, the only proper
  normal subgroups of $G$ are those of~$\Z{G}$.  If $\rho$ is an
  ordinary representation of~$G$ then $\rho(\Z{G})$ is a pro-finite
  subgroup of a Lie group and thus finite.
  This means that $\ker\rho=\ker(\rho|_{\Z{G}})$ is co-finite in
  $\Z{G}$, and $G/\ker\rho$ is an extension of Lie groups (namely
  $G/\Z{G}\cong\PSL[2]\RR$ and the finite group $\Z{G}/\ker\rho$). Now
  $G/\ker\rho$ is a connected Lie group, has the same Lie algebra as
  $\PSL[2]\RR$ and possesses a faithful ordinary representation. Thus
  $|\Z{G}/\ker\rho|\le2$ and $\ker\rho$ contains
  $\smallset{z^2}{z\in\Z{G}}\cong\prod_{p\in\PP}\ZZ_p$.
  Since $G/\smallset{z^2}{z\in\Z{G}}\cong\SL[2]\RR$ clearly has a
  faithful ordinary representation we obtain
  $\KO{G}\cong\prod_{p\in\PP}\ZZ_p$. 
\end{exam}

\begin{defs}
  A topological group $G$ is called \emph{monothetic} if there exists
  $g\in G$ such that the closure of the subgroup generated by~$g$ is
  dense in~$G$. Any such $g$ is called a \emph{topological generator}
  of~$G$. The class of all compact monothetic (necessarily abelian)
  groups will be denoted by~$\cMon$.

  The one-parameter subgroups of $G$ are the elements of
  $\Hom{\RR}{G}$. We say that $G$ has a \emph{dense one-parameter
    subgroup} if there is $\varphi\in\Hom{\RR}{G}$ with
  $\closure{\varphi(\RR)}=G$.
\end{defs}

\begin{lemm}\label{lemm:monothetic}
  \begin{enumerate}
  \item A locally compact monothetic group is either compact or 
    isomorphic to~$\ZZ$.
  \item If $A\in\LCA$ has a dense one-parameter subgroup then $A$ is
    either isomorphic to~$\RR$ or $A$ is a connected compact
    monothetic group.
  \item\label{dualMonothetic} The Pontryagin duals of compact monothetic groups are the
    subgroups of the discrete group
    $\QQ^{(2^{\aleph_0})}\times\QQ/\ZZ$.
  \item The duals of connected compact monothetic groups are the
    subgroups of the discrete group\/~$\QQ^{(2^{\aleph_0})} \cong
    \RR_{\mathrm{discr}}$.
  \end{enumerate}
\end{lemm}
\begin{proof}
  The first two assertions are known as Weil's Lemma,
  cf.~\cite[6.26]{MR2226087}.
  Assertion~\ref{dualMonothetic} is due to~\cite{MR0006543}, we
  present the argument for the reader's convenience.
  The existence of a cyclic subgroup means that there is an
  epimorphism $\eta\colon\ZZ\to A$ in the category $\LCA$ which upon
  dualizing gives a monomorphism
  $\dual{\eta}\colon\dual{A}\to\dual{\ZZ}\cong\RR/\ZZ$;
  cf.~\cite[15.5, 15.7, 20.13]{MR2226087}. Since $A$ is compact the
  dual $\dual{A}$ is discrete (see.~\cite[20.6]{MR2226087}), and we
  may interpret $\dual{\eta}$ as an embedding of $\dual{A}$ into the
  discrete group $\RR_{\mathrm{discr}}/\ZZ \cong
  \QQ^{(2^{\aleph_0})}\times\QQ/\ZZ$.  Now connectedness of~$A$
  implies that $\dual{A}$ is torsion-free
  (cf.~\cite[23.18]{MR2226087}) and $\dual{\eta}$ induces an embedding
  of $\dual{A}$ into the quotient $\QQ^{(2^{\aleph_0})}$ of
  $\RR_{\mathrm{discr}}/\ZZ$ modulo its torsion group $\QQ/\ZZ$.

  In order to prove the last assertion we dualize the epimorphism
  $\varphi\colon\RR\to A$ to a monomorphism
  $\dual\varphi\colon\dual{A}\to\dual{\RR}\cong\RR$ and then replace
  the range by the discrete group $\RR_{\mathrm{discr}}$.
\end{proof}

\begin{exam}\label{exa:HeisenbergWithCpCenter}
  Generalizing the construction described in~\ref{realHeis} we take
  $A\in\LCA$ with a dense one-parameter group $\varphi\in\Hom{\RR}{A}$
  and define a multiplication~$*$ on $\RR^2\times A$ by
  $(a,b,x)*(c,d,y)=(a+c,b+d,x+y+\varphi(ad-bc))$. Then
  $H_\varphi:=(\RR^2\times A,*)$ is a connected locally compact
  group. If $A$ is compact we proceed as in~\ref{realHeis} to see
  $\KO{H_\varphi}=\{0\}^2\times A\cong A$. 
\end{exam}

\begin{exam}\label{exa:monotheticConnCA}
  Again, let $S$ denote the universal covering group of
  $\SL[2]{\RR}$. Assume that $A$ is a compact monothetic group and
  let~$c$ be a topological generator. Proceeding as
  in~\ref{SL2timesH} we construct the quotient $G$ of $S\times A$ such
  that $\zeta^2$ is identified with~$c$.  Then an argument as
  in~\ref{SL2timesH} shows that $\KO{G}$ contains the closure of~$c$
  in~$G$. This is a subgroup isomorphic to~$A$.
\end{exam}

\section{Connected pro-Lie groups}

In order to show $\CA\subset\KO{\Conn\ProLie}$ we
study free compact abelian groups. 
  
\begin{defs}\label{defs:weight}
  For a pointed compact space $X$ let $X/\text{conn}$ be the totally
  disconnected compact space of connected components of~$X$ and let
  $\wt(X)$ denote the weight of $X$ (i.e., the minimal cardinality of
  a basis for the topology on~$X$). Put %
  $\wt_0(X) := \wt(X)-1$; this coincides with $\wt(X)$ if the latter
  is infinite.
  
  Let $C_0(X,\TT)$ denote the set of all continuous functions from $X$
  to~$\TT$ mapping the base point to~$0$. This set endowed with the
  compact-open topology and the pointwise operations becomes a
  topological group. The quotient $\left[X,\TT\right] := C_0(X,\TT) /
  C_0(X,\TT)_0$ modulo the connected component $C_0(X,\TT)_0$ is
  discrete (cf. the paragraph preceding~\cite[8.50]{MR2261490}); its
  compact Pontryagin dual $\dual{\left[X,\TT\right]}$ plays a crucial
  role in the structure of the free compact abelian group $F(X)$.
\end{defs}
  
\begin{theo}\label{structureFreeCA}
  For every nonsingleton pointed compact space the free
  compact abelian group $F(X)$ is isomorphic to
  $(\dual{\QQ})^{\wt(X)^{\aleph_0}} \times
  \prod_{p\in\PP}\ZZ_p^{\wt_0(X/\text{conn})} \times
  \dual{\left[X,\TT\right]}$.  The group $\left[X,\TT\right]$ is
  torsion-free, and its Pontryagin dual\/ $\dual{\left[X,\TT\right]}$ is
  a quotient of some power of\/~$\dual\QQ$.
\end{theo}
\begin{proof}
  The structure of $F(X)$ is known from~\cite[1.5.4]{MR849093},
  cf.~\cite[8.67]{MR2261490}.  For every compact Hausdorff space
  $X$ the group $\left[X,\TT\right] \cong H^1(X,\ZZ)$ is torsion-free,
  see~\cite[1.3.1]{MR849093}, cf.~\cite[8.50(ii)]{MR2261490}. For
  $d:=\dim_\QQ(\QQ\otimes\left[X,\TT\right])$ we have an embedding of
  $\left[X,\TT\right]$ into
  $\QQ\otimes\left[X,\TT\right]\cong\QQ^{(d)}$ which dualizes to the
  quotient map in question.
\end{proof}

\begin{theo}\label{KOconnProLie}
  The class $\KO{\Conn\ProLie}$ contains at least the class $\aP(\CgAL
  \cup\CA) = \aP({\{\ZZ,\RR\}\cup\CA})$. In particular, we have
  $\Conn{\Ab\ProLie} \subset \KO{\Conn\ProLie}$.

  The class $\KO{\Conn\LCG}$ contains the class $\Small\LCA =
  \fP(\{\ZZ,\RR\}\cup\cMon)$ consisting of all groups of the
  form $\ZZ^f\times A\times\RR^e$ with $e,f\in\NN$ and a compact
  monothetic group~$A$.
\end{theo}

\begin{proof}
  Abelian pro-Lie groups are studied in~\cite{MR2103546},
  cf.~\cite[Ch.\,5]{MR2337107}: in particular, the connected ones are
  of the form $\RR^c\times C$ where $c$ is arbitrary (possibly
  infinite) and $C$ is a connected compact abelian group.  Such a
  group belongs to $\Conn\LCA$ precisely if~$c$ is finite,
  cf.~\cite[24.9]{MR2226087}.

  Let $C$ be a compact abelian group. Then $C$ is a quotient of the
  free compact abelian group $F(C)$. From~\ref{structureFreeCA} we
  know that $F(C)$ is a quotient of a product of compact monothetic
  groups. The class~$\cMon$ of all compact monothetic groups is
  contained in $\KO{\Conn\LCG} \subset \KO{\Conn\ProLie}$
  by~\ref{exa:monotheticConnCA} and~\ref{exa:HeisenbergWithCpCenter}.
  From~\ref{cartProd} and~\ref{quotientsKOConnProLie} we conclude $\CA
  = \hQ\aP(\cMon) \subseteq \KO{\Conn\ProLie}$ and
  $\hQ\fP(\cMon) \subseteq \KO{\Conn\LCG}$.
  
  From~\ref{SL2timesH} we recall that $\ZZ$ and $\RR$ lie in
  $\KO{\Conn\LCG}\subset\KO{\Conn\ProLie}$; it remains to use products
  once again.
\end{proof}

\begin{rema}
  The factor~$A$ in~\ref{KOconnProLie} cannot be arbitrarily large; we
  show $\wt(A_0)\le 2^{\aleph_0}$ in~\ref{connKOsmall} below. %
  If the answer to~\ref{sKOfinite} is affirmative then we know that
  $\Small\LCA$ coincides with~$\KO{\Conn\LCG}$,
  cf.~\ref{KorConnLCGfinGen}.
\end{rema}

\begin{rems}
  We have mentioned in~\ref{structureFreeCA} that the factor
  $\dual{\left[X,\TT\right]}$ of $F(X)$ has a torsion-free dual.
  Conversely, every torsion-free abelian group $A$ is isomorphic to
  $\left[X,\TT\right]$ for some compact connected Hausdorff space
  (namely, for the underlying space $X=|\dual{A}|$ of the Pontryagin
  dual of~$A$), see~\cite[1.3.2]{MR849093}. The corrections
  in~\cite{MR927286} only concern assertions about cardinalities
  (dimension, rank) in~ \cite{MR849093}.
\end{rems}

\begin{prop}\label{noLargeDiscrete}
  Each discrete member of\/
  $\KO{\Conn\ProLie}$ is finitely generated. 
  Each member of\/
  $\KO{\Conn\LCG}$ is compactly generated.
\end{prop}
\begin{proof}
  Discrete central subgroups of connected pro-Lie groups are finitely
  generated by~\cite[5.10]{MR2475971}.
  Consider $G\in\Conn\LCG$. Any compact neighborhood in $\KO{G}$
  generates a central subgroup~$C$ of~$G$, and $\KO{G}/C$ is a
  discrete subgroup of $G/C \in \Conn\LCG \subset \Conn\ProLie$.  Now
  it remains to note that the class of compactly generated locally
  compact groups is closed under extensions,
  see~\cite[6.11]{MR2226087}.
\end{proof}

\begin{rema}\label{cpGenLCA}
  Locally compact abelian groups are compactly generated precisely if
  they are contained in some almost connected locally compact group,
  see~\cite[Cor.\,1]{MR2542208}.  The compactly generated members of
  $\LCA$ are of the form $\ZZ^f\times C \times\RR^e$ with natural
  numbers $e,f$ and some $C\in\CA$, cf.~\cite[23.11]{MR2226087}. %
  Note that \emph{every} $C\in\CA$ is contained in a connected compact
  group because the characters separate the points: this yields an
  embedding into~$\TT^{\dual{C}}$.
\end{rema}

\goodbreak
\section{Connected locally compact groups}

Our aim in this section is to establish a bound on the size of
$\KO{G}$ if $G\in\Conn\LCG$. To this end, we need some information on
the weight (cf.~\ref{defs:weight}) and generating rank. 

  If~$X$ is discrete then
  $\wt(X)$ is just the cardinality~$|X|$.  The function $\wt$ is
  monotonic; i.e. each subspace $Y\subseteq X$ satisfies $\wt(Y)\le
  \wt(X)$.  For an infinite compact group the weight coincides with the
  \emph{local weight}, i.e., the minimal cardinality of a neighborhood
  basis. %

  For~$C\in\Conn\CA$ a finer invariant than the weight is the \emph{rank}
  $\dim_\QQ(\QQ\otimes\dual{C})$ of its dual~$\dual{C}$. %
  Note that $\dual{C}$ embeds in $\QQ\otimes\dual{C}$ only if
  $\dual{C}$ is torsion-free (i.e., if $C$ is connected,
  see~\cite[23.18]{MR2226087}). The rank of~$\dual{C}$ coincides with
  the topological dimension of~$C$, cf.~\cite[8.26]{MR2261490}. %
  Every compact abelian Lie group has finite dimension; its dual is
  finitely generated.

\begin{lemm}\label{exas:Weight}
  \begin{enumerate}
  \item For each positive integer~$n$ we have $\wt(\RR^n)=\aleph_0$.
  \item\label{weightDual} The equality $\wt(A)=\wt(\dual{A})$ holds for each $A\in\LCA$.
  \item In particular, for $C\in\CA$ we have $\wt(C)=|\dual{C}|$.
  \item For $C\in\CA$ we have $\wt(C_0) = \max\{\aleph_0,\dim_\QQ(\QQ\otimes\dual{C})\}$
    unless $C$ is trivial.
  \item If\/ $n$ is a positive integer and $C\in\Conn\CA$ then
    $\wt(\RR^n\times C) = \max\{\aleph_0,\wt(C)\}$.
  \item If\/ $G\in\CG$ and\/~$N$ is a totally disconnected closed
    normal (and thus central) subgroup of\/~$G_0$ then $\wt(G)=\wt(G/N)$.%
  \end{enumerate}
\end{lemm}
\begin{proof}
  The assertion on $\wt(\RR^n)$ is obvious from the fact that the
  underlying space is metrizable and the weight equals the local
  weight.  See~\cite[24.14]{MR551496} or~\cite[7.76]{MR2261490} for
  assertion~\ref{weightDual}.  The assertions on compact groups are
  taken from~\cite[12.25]{MR2261490} and~\cite[3.2]{MR1082789}.
\end{proof}

\begin{lemm}\label{characterizeMonothetic}
  For $C\in\CA$ we have $\wt(C_0)\le 2^{\aleph_0}$ precisely if $C_0$
  is monothetic. %
  If\/~$C\in\CA$ has a finitely generated dense subgroup then
  $\wt(C)\le2^{\aleph_0}$. 
\end{lemm}
\begin{proof}
  From~\ref{exas:Weight} we know
  $\wt(C_0)=\dim_\QQ(\QQ\otimes\dual{C_0})$. Since $\dual{C_0}$ is
  torsion-free we have an embedding
  $\eta\colon\dual{C_0}\to\QQ\otimes\dual{C_0}$. Thus $\dual{C_0}$ is
  the dual of a monothetic group, see~\ref{lemm:monothetic}. 
\end{proof}

\begin{ndef}[Suitable sets and generating rank]
  Following~\cite{MR1082789} (cf.~\cite[12.1]{MR2261490}), a subset
  $X$ of a topological group $G$ is called \emph{suitable} if it does
  not contain the neutral element~$1$, is discrete and closed in
  $G\setminus\{1\}$, and generates a dense subgroup of~$G$. Every
  locally compact group possesses suitable sets
  by~\cite[1.12]{MR1082789}.
  We define the \emph{generating rank} $s(G)$ as the minimum over
  the cardinalities of suitable sets.
\end{ndef}

%\enlargethispage{10mm}%%%%%%%%%%%%%%%%%%%%%%%%%%%%%%%%%%%%%%%%%%%%%%%%%%
\begin{exas}\label{exas:genRank}
  \begin{enumerate}
  \item For the additive group~$\RR$ any two elements that are
    linearly independent over~$\QQ$ form a suitable set. Thus
    $s(\RR)=2$.
  \item A suitable set for $\RR^n$ needs at least $n+1$ elements
    because fewer vectors will either be linearly dependent (and thus
    contained in a proper closed subgroup) or form a basis (and then
    generate a discrete proper subgroup).  It is known that there
    exists $v\in\RR^n$ such that $v+\ZZ$ generates a dense subgroup of
    $\RR^n/\ZZ^n$.  The standard basis together with any such~$v$
    forms a suitable set for~$\RR^n$. Thus $s(\RR^n)=n+1$.
  \item
    For $G\in\Conn\CG$ the generating rank depends on the weight,
    see~\cite[12.22]{MR2261490}:
    \begin{enumerate}
    \item If $\wt(G)\le2^{\aleph_0}$ and $G$ is abelian (but not
      trivial) then $s(G)=1$.
    \item If $\wt(G)\le2^{\aleph_0}$ and $G$ is not abelian then
      $s(G)=2$.
    \item If $\wt(G)>2^{\aleph_0}$ then
      $s(G)=\min\set{\beth}{\wt(G)\le\beth^{\aleph_0}}$.
    \end{enumerate}
  \item In particular, the generating rank of any connected Lie group
    is finite. %
  \end{enumerate}
\end{exas}

\goodbreak
\begin{lemm}[{\cite[12.26]{MR2261490}}]\label{genRankModTotDisconnected}
  Let $G\in\Conn\CG$ %
  and let $N$ be a normal subgroup of~$G$. Then $s(G/N)\le s(G)\le
  s(G/N)+s(N)$. If $N$ is totally disconnected then $s(G)=s(G/N)$.
\qed
\end{lemm}

\begin{ndef}[Reduction to maximal compact subgroups]\label{redMaxCp}
  Let $G\in\almConn\LCG$. By the Mal{\cprime}tsev--Iwasawa Theorem
  (see~\cite[Thm.\,13]{MR0029911} combined with the solution of
  Hilbert's Fifth Problem~\cite{MR0058607}, cf.~\cite{MR0073104}
  or~\cite{MR0276398}) %
  there exists a maximal compact subgroup $M$ of~$G$ which is unique
  up to conjugacy. If $M\ne G$ then there is a positive integer~$n$ %
  (called the characteristic index in~\cite{MR0029911}) %
  and there are subgroups
  $R_1,\dots,R_n$ all isomorphic to~$\RR$ such that $G=R_1\cdots
  R_nM$.  The number~$n$ equals the topological dimension $\dim{G/M}$
  of the coset space $G/M$ (which is actually a manifold).  If $M=G$
  we just have $n=0$.

  Note that $M$ is connected. Thus the weight of $G$ equals that
  of~$M$ if $G=M$ and satisfies $\wt(G)=\max\{\aleph_0,\wt(M)\}$
  otherwise.
\end{ndef}

\begin{prop}\label{genRankAlmConnLCG}
  For\/ $G\in\almConn\LCG$ pick a maximal compact subgroup\/~$M$ 
  of\/~$G$. Then $s(G)\le 2\dim{G/M}+s(M)$.
  In particular, the generating rank $s(G)$ is finite if\/
  $\wt(M)\le2^{\aleph_0}$. %
  If\/ $\psi\colon G\to H$ is a continuous homomorphism with dense range
  then $s(G)\ge s(H)$. 
\end{prop}

\begin{proof}
  We use subgroups $R_1,\dots,R_n$ as in~\ref{redMaxCp}, where
  $n:=\dim{G/M}$.  Combining a suitable set for~$M$ with suitable sets
  for each~$R_j$ we find $s(G)\le 2n+s(M)$. From~\ref{exas:genRank} we
  thus infer that $s(G)$ is finite if $G\in\Conn\LCG$ satisfies
  $\wt(G)\le2^{\aleph_0}$.

  Every suitable set~$X$ constructed in this way will be relatively
  compact because only finitely many elements lie outside the compact
  group~$M$. Now~\cite[12.4]{MR2261490} asserts that
  $\psi(X)\setminus\{1\}$ will be suitable in~$H$ whenever $\psi\colon
  G\to H$ is a continuous homomorphism with dense range.   
\end{proof}

\begin{lemm}\label{estimateGenRankVsWeight}
  For each non-trivial $G\in\almConn\LCG$ we have $s(G)\le \wt(G)\le
  s(G)^{\aleph_0}$.
\end{lemm}
\begin{proof}
  By the Mal{\cprime}tsev--Iwasawa Theorem (cf.~\ref{redMaxCp}) %
  we know that $G$ is homeomorphic to $\RR^n\times M$ for a maximal
  compact subgroup~$M$ and some nonnegative integer~$n$. The estimates
  $s(M)\le \wt(M)\le s(M)^{\aleph_0}$ are valid for any non-trivial
  compact group $M$, see~\cite[12.27]{MR2261490}.

  If the group $G$ is compact it coincides with~$M$. If $M$ is trivial
  but $n\ge1$ then~\ref{genRankAlmConnLCG} yields that $s(G)\le 2n$ is
  finite. Thus $s(G)\le \aleph_0 = \wt(G) < 2^{\aleph_0} =
  s(G)^{\aleph_0} $.
  
  There remains the case where $G$ is not compact and $M$ is not
  trivial. Then $n\ge1$ and $\wt(G)=\max\{\aleph_0,\wt(M)\}$. Now the
  estimates $s(G)\le 2n+s(M)\le \aleph_0+s(M) \le \aleph_0+\wt(M) = \wt(G)
  $ from~\ref{genRankAlmConnLCG} and $\wt(G) = \aleph_0+\wt(M) \le
  \aleph_0+s(M)^{\aleph_0} \le 2^{\aleph_0}+ s(M)^{\aleph_0} =
  s(G)^{\aleph_0}$ yield the claim.
\end{proof}

\begin{rema}
  The inequality $s(G)\le \wt(G)$ holds for arbitrary $G\in\LCG$
  see~\cite[4.2]{MR1082789}. %
\end{rema}

\begin{lemm}\label{IwasawaRevisited}
  Let $G\in\Conn\LCG$. We consider a compact normal subgroup $K$
  of~$G$, the centralizer $C:=\C[G]{K}$ and its connected
  component~$C_0$.  Then $G=C_0K$.
\end{lemm}
\begin{proof}
  We abbreviate $D:=C\cap K$. From~\cite[Thm.\,2]{MR0029911} we know
  $CK=G$. The group $C$ is a closed subgroup of the $\sigma$-compact
  group $G$ and thus $\sigma$-compact itself, see~\cite[6.10,
  6.12]{MR2226087}. The Open Mapping Theorem
  (cf.~\cite[6.19]{MR2226087}) yields that $C/D$ is isomorphic
  to~$G/K$. Now $C/(C_0D)$ is totally disconnected
  (cf.~\cite[6.9]{MR2226087}) but also connected because it is a
  continuous image of $C/D \cong G/K$. Thus $C=C_0D$ and $C_0K = C_0DK
  = CK = G$.
\end{proof}

\begin{theo}\label{sKOfiniteIfAbK}
  If\/ $G\in\Conn\LCG$ has a compact normal subgroup $K$ such that
  $G/K$ is a Lie group and $K_0$ is abelian then $s(\closure{G'})$ is
  finite. Consequently, the group $\KO{G}$ satisfies $\wt(\KO{G})\le
  2^{\aleph_0}$ in this case.
\end{theo}
\begin{proof}
  The compact abelian normal subgroup $K_0$ is contained in the center
  of~$G$. A dense subgroup of 
  $\closure{G'}$ is generated by $\gamma(S\times S)$ where $S$ is any
  suitable set for $G/K_0$ and $\gamma\colon G/K_0\times G/K_0 \to
  \closure{G'}\colon (aK_0,bK_0) \mapsto aba^{-1}b^{-1}$ is induced
  by the commutator map.
  
  We pick a maximal compact subgroup~$M$ of~$G$. Then
  $K_0\le M$ and by~\ref{genRankAlmConnLCG} there exists a nonnegative
  integer~$n$ such that $s(G/K_0) \le 2n+s(M/K_0)$. Now $s(G/K_0)$ is finite
  because the compact connected Lie group $M/K$ has finite generating
  rank and $s(M/K_0)=s(M/K)$ by~\ref{genRankModTotDisconnected}.

  Thus we may choose a finite set $S$; then $\gamma(S\times
  S)\setminus\{1\}$ is a finite (and thus indeed closed and discrete)
  suitable set for~$\closure{G'}$. The bound for the weight of
  $\KO{G}$ now follows from~\ref{estimateGenRankVsWeight},
  monotonicity of the weight function and the fact that $\KO{G}$ is
  contained in $\closure{G'}$, see~\ref{KOinZcapComm}.
\end{proof}

  For any solvable compact group the connected component is
  abelian, see~\cite[Thm.\,2]{MR0029911}. Thus~\ref{sKOfiniteIfAbK}
  yields: 

\begin{coro}\label{sKOfiniteIfSolv}
  If\/ $G$ is a solvable connected locally compact group then
  $s(\closure{G'})<\aleph_0$ and\/
  $\wt(\KO{G})\le\wt{\closure{G'}}\le 2^{\aleph_0}$.  
\qed
\end{coro}

\begin{exam}\label{exam:wZlarge}
  Let $\aleph$ be any cardinal, let $S$ be a connected compact Lie
  group with simple Lie algebra and non-trivial center~$Z$, and
  put $G := S^\aleph$. Then $Z$ is finite and
  $\closure{S'} = S$. Thus $G$ is a compact connected group with
  $\closure{G'}=G$ and $\Z{G}\cap\closure{G'} = \Z{G} \cong
  Z^\aleph$.  Now $\wt(\Z{G}\cap\closure{G'}) =
  \wt(Z^\aleph) = |Z^{(\aleph)}|$, and this cardinality may be
  arbitrarily large.
\end{exam}

\begin{prop}\label{connKOsmall}
  For each $G\in\Conn\LCG$ we have $\wt(\KO{G}_0)\le2^{\aleph_0}$.
  In particular, the connected component of the maximal compact
  subgroup of $\KO{G}$ is monothetic. 
\end{prop}
\begin{proof}
  Let $K$ be the maximal compact normal subgroup of~$G$
  (cf.~\cite[Thm.\,14]{MR0029911}). Then~\ref{IwasawaRevisited} says
  $G=C_0K$ where $C_0$ is the connected component of the centralizer
  $C:=\C[G]{K}$ of~$K$ in~$G$. Thus $\closure{G'} = \closure{C_0'K'} =
  \closure{C_0'}\,\closure{K'}$ because $\closure{K'}$ is compact, and
  $(\Z{G}\cap\closure{G'})_0 \le C_0\cap\closure{G'}$ is contained in
  $\closure{C_0'}(C_0\cap\closure{K'}) \le C_0$.

  The locally compact group $C_0/(C_0\cap K)$ admits an injective
  continuous homomorphism into the Lie group~$G/K$ and thus is a Lie
  group itself. Moreover, we have that $C_0\cap K$ is contained in the
  center of~$K$ and thus abelian. Thus~\ref{sKOfiniteIfAbK} applies
  to~$C_0$ and~$K$, yielding $\wt(\closure{C_0'})\le2^{\aleph_0}$. %

  The group $B:=C_0\cap\closure{K'}$ is contained in the intersection
  of~$\closure{K'}$ with the center of the compact group~$K$. Thus $B$
  is totally disconnected, see~\cite[9.23]{MR2261490}, %
  and the quotient $B\closure{C_0'}/\closure{C_0'} \cong
  B/(B\cap\closure{C_0'})$ is totally disconnected, as well.  This
  yields that the connected component $\KO{G}_0$ of $\KO{G}$ is
  contained in $\closure{C_0'}$, and the bound
  $\wt(\KO{G}_0)\le2^{\aleph_0}$ is established. %
  Now the connected component of the maximal compact subgroup of
  $\KO{G}$ is monothetic by~\ref{characterizeMonothetic}.
\end{proof}

\begin{theo}\label{KorConnLCGfinGen}
  For each $G\in\Conn\LCG$ there exist $A\in\CA$ and natural
  numbers~$e,f$ such that $A_0$ is monothetic and
  $\KO{G}\cong\ZZ^f\times A\times\RR^e$.
  In particular, the dimension of members of\/~$\KO{\Conn\LCG}$ is
  bounded by~$2^{\aleph_0}$. 
\end{theo}
\begin{proof}
  We combine~\ref{connKOsmall} with~\ref{noLargeDiscrete},
  \ref{cpGenLCA}, and~\ref{characterizeMonothetic}. 
\end{proof}

\goodbreak
\section{Partial results}

The classes $\KO{\SepLie}$, $\KO{\Lie}$, $\KO{\LCG}$
and~$\KO{\ProLie}$ are quite large and not very well understood. We
will indicate some large subclasses and note
(in~\ref{KOLCGnotClosedUnderQ}) that these classes are not closed
under the operations~$\cS$, $\hQ$, and~$\cQ$.

\goodbreak
\begin{theo}\label{KOsepLie}
  The class $\KO{\SepLie}$ contains the following:
  \begin{enumerate}
  \item The class $\KO{\Conn\Lie} = \KO{\almConn\Lie} = \CgAL$,
    cf.~{\upshape\ref{KOconnLie}} and~{\upshape\ref{almConnLie}}.
  \item\label{SimpleCharP} All countable discrete simple
    groups with infinite elementary abelian subgroups, 
    cf.~{\upshape\ref{ex:Burnside}\ref{bounded}}. 
  \item Countably infinite discrete simple groups with finitely many
    conjugacy classes, such as those constructed as HNN-extensions,
    see~{\upshape\ref{ex:Burnside}\ref{HNN}}.
  \end{enumerate}
  Moreover, $\KO{\SepLie}$ is closed under~$\fP$.
\end{theo}

Note that $\KO{\SepLie}$ is considerably larger than
$\KO{\almConn\Lie} = \CgAL$. 

\goodbreak
\begin{theo}\label{KOLCG}
  The class $\KO{\LCG}$ contains the following:
  \begin{enumerate}
  \item The class $\KO{\SepLie} \cup \KO{\Conn\LCG}$, and thus
    $\CgAL$ and all compact monothetic groups.
  \item  All groups of the form $C^D$
    where $C$ is a compact simple non-abelian group and $D$ is
    infinite, see~{\upshape\ref{powerD}}.
  \item\label{pAdicSimple}  All simple non-discrete
    totally disconnected locally compact groups.
  \item\label{tooLargeSimple}  All simple discrete groups
    of cardinality larger than~$2^{\aleph_0}$.
  \item
    All discrete simple groups with infinite elementary abelian
    subgroups. 
  \end{enumerate}
  Moreover, $\KO{\LCG}$ is closed under~$\fP$ but not under~$\aP$.
\end{theo}

The class in~\ref{KOLCG}\ref{pAdicSimple} includes the simple $p$-adic
groups such as $\PSL[n]{\QQ_p}$.  Among the groups
in~\ref{KOLCG}\ref{tooLargeSimple} we find, for instance, the simple classical
groups over large fields such as $\PSL[n]{F}$ where $F=\QQ(X)$ is a
purely transcendental extension with a transcendency basis $X$ such
that $|X|>|\RR|$.

\enlargethispage{8mm}

\begin{prop}\label{noSOthree}
  The groups $\SO3{\RR}$ and\/ $\PSL[2]{\CC}$ do not belong to $\KO{\TG}$. 
\end{prop}
\begin{proof}
  The Lie algebra $\lie{\SO3\RR}$ is isomorphic to the vector product
  algebra $(\RR^3,\times)$, and $\Ad(\SO3\RR)\cong\SO3\RR$ contains
  all automorphisms of that algebra. Therefore, each automorphism of
  $\SO3\RR$ is an inner automorphism, see\footnote{
  The discussion of $\Aut{\lie{\SO3\RR}}$
  in~\cite[p.\,252]{MR2261490} contains an error; indeed
  $\Orth3\RR\setminus\SO3\RR \not\subseteq \Aut{\lie{\SO3\RR}}$.}%
  ~\cite[6.59]{MR2261490}. 

  Now assume that there exists a group $G\in\TG$ with
  $\KO{G}\cong\SO3\RR$. Then $G$ is the direct product of $\KO{G}$
  with its centralizer~$C$, see~\ref{IwasawaRevisited}. Thus
  $G/C\cong\SO3{\RR}$ has a faithful ordinary representation, and
  $\KO{G}\le C$. This is a contradiction.
  
  For the group $\PSL[2]{\CC}\cong\PGL[2]{\CC}$ we can proceed in the
  same way because this group also has only inner automorphisms,
  see~\cite{54.0149.02}, cf.~\cite{MR606555}. 
\end{proof}

\begin{coro}\label{KOLCGnotClosedUnderQ}
  The classes $\KO{\LCG}$ and $\KO{\ProLie}$ are not closed under any
  one of the operations~$\cS$, $\hQ$, or~$\cQ$.  
\qed
\end{coro}

\begin{rema}
  The group $\PSL[2]{\RR}$ has outer automorphisms (induced by
  elements of $\GL[2]{\RR}$ with non-square determinant). Every group
  $\PSL[n]{F}$ with $n>2$ over a commutative field~$F$ has outer
  automorphisms induced by polarities of the projective space. Thus
  the argument used in the proof of~\ref{noSOthree} does not easily
  extend to arbitrary classical simple groups.
\end{rema}

\goodbreak
\section{Open questions}
\label{sec:openQuestions}

\begin{prob}\label{sKOfinite}\label{prob:ConnLCG}
  Is it true that %
  $\wt(\KO{G})\le 2^{\aleph_0}$ holds for every $G\in\Conn\LCG$~? %
  \comments %
  If the answer to this problem is affirmative then $\KO{\Conn\LCG} =
  \Small\LCA$, cf.~\ref{KOconnProLie}, \ref{noLargeDiscrete} and~\ref{cpGenLCA}. %
  
  In~\ref{noLargeDiscrete} we have seen that $\KO{\Conn\LCG}$ consists
  of compactly generated groups, and~\ref{connKOsmall} says that the
  connected component of the maximal compact subgroup of $\KO{G}$ lies
  in $\Small\CA$.  For an affirmative answer to~\ref{prob:ConnLCG} it
  suffices to exclude totally disconnected compact groups $A$ with
  $\wt(A) > 2^{\aleph_0}$ from $\KO{\Conn\LCG}$ because that class is
  closed under the passage to Hausdorff quotients, cf.~\ref{quotients}.

  For $G\in\Conn\LCG$ consider the maximal compact normal subgroup~$K$
  of~$G$ and let~$C_0$ be the connected component of the centralizer
  of~$K$. Then the weight of~$C_0\cap\closure{K'}$ may be arbitrarily
  large, as~\ref{exam:wZlarge} shows.  It is therefore clear that we
  have to find a more subtle approach than~\ref{KOinZcapComm} if we
  want to give an affirmative answer to~\ref{sKOfinite}.
\end{prob}

\begin{prob}\label{prob:KOConnProLie}
  Which abelian pro-Lie groups are in $\KO{\Conn\ProLie}$~?  %
  \comments %
  Every element of $\KO{\Conn\ProLie}$ is contained in the center of a
  connected pro-Lie group (namely,~$G$) and thus contained in some
  connected abelian pro-Lie group, cf.~\cite[12.90]{MR2337107}.
  From~\ref{KOconnProLie} we thus infer 
  \[
  \Conn{\Ab\ProLie} \subset \KO{\Conn\ProLie} \subseteq
  \cS(\Conn{\Ab\ProLie}) \,.
  \]
  A discrete abelian group belongs to $\KO{\Conn\ProLie}$ precisely if
  it is finitely generated (and thus lies in~$\CgAL$),
  see~\ref{noLargeDiscrete}.  
\end{prob}

\begin{probs}\quad
  \begin{enumerate}
  \item Is $\KO{\Conn\ProLie}$ closed under~$\cQ$~?
  \item Is $\KO{\Conn\ProLie}$ closed under~$\cS$~?
\end{enumerate}
\comments %
We know that $\KO{\Conn\ProLie}$ is not closed under~$\hQ$ because the
group $\RR^\RR$ belongs to $\KO{\Conn\ProLie}$ but has a quotient
which is not complete (cf.~\cite[4.11]{MR2337107}), and thus does not
lie in $\KO{\Conn\ProLie} \subseteq \Ab\ProLie$.  %

From~\ref{quotientsKOConnProLie} we know that $\KO{\Conn\ProLie}$ is
closed under quotients modulo locally compact groups. %
\end{probs}

\begin{prob}
  What about $\KO{G}$ if $G$ is an almost connected pro-Lie group, or
  an almost connected locally compact group?
\comments
The conclusion of~\ref{connCenter} breaks down if we drop the
assumption of connectedness, cf.~\ref{powerD}.  In many questions
about the structure of pro-Lie groups it is possible to weaken a
connectedness hypothesis to ``almost connected'' (i.e., compactness of
$G/G_0$). For instance, almost connected locally compact groups are
pro-Lie groups while for arbitrary disconnected locally groups the
homomorphisms into Lie groups need not separate the points.  We have
seen in~\ref{almConnLie} that $\KO{\almConn\Lie} = \KO{\Conn\Lie}$ is very
well behaved.

The examples in~\ref{powerD} fail to be almost connected,
and also fail to be pro-Lie groups. The same applies to most of our
examples of groups $G\in\LCG$ with $\KO{G}=G$.
Note that the discrete examples are in $\Lie\subset\ProLie$. 
\end{prob}

\begin{probs}
  For $\mathcal{G}\in\{\LCG,\ProLie\}$ we ask: 
  \begin{enumerate}
  \item Is $\CA$ completely contained in $\KO{\mathcal{G}}$~?
  \item Is $\LCA$ completely contained in $\KO{\mathcal{G}}$~?
  \item Which part of $\CG$ is contained in $\KO{\mathcal{G}}$~?
  \item Which discrete groups are in $\KO{\mathcal{G}}$~?
  \item What \emph{is} $\KO{\mathcal{G}}$~?
\end{enumerate}
\comments %
If we drop all connectedness assumptions on~$G$ we obtain
examples $G$ where $\KO{G}$ is not abelian. However, the inclusion
$\KO{\mathcal{G}}\subset\mathcal{G}$ falls far from being an
equality. The class $\KO{\mathcal{G}}$ appears to be complicated and
not accessible to an easy ``constructive'' description (such as:
``take the following basic examples and use certain constructions like
products or quotients''). 

Definitely, the class $\CG$ is not completely contained in
$\KO{\mathcal{G}}$. For instance, we know that
$\CG\cap\KO{\mathcal{G}}$ is not closed under~$\cS$ or~$\hQ$,
see~\ref{noSOthree}.

It is also open whether arbitrary discrete groups are in
$\KO{\ProLie}$. 
\end{probs}

\enlargethispage{8mm}
\section{Appendix: linear groups}
\label{sec:appendix}

We collect some known facts regarding the question whether a given
group is \emph{linear}, i.e., admits a faithful ordinary
representation. In our present terminology, this means that the
trivial group is a member of $\smallset{\ker\rho}{\rho\in\OR{G}}$.
Examples like the additive group~$\ZZ_p$ of $p$-adic integers or any
infinite elementary abelian group show that this condition is, in general, 
much stronger than the condition that $\KO{G}$ is trivial.

By way of contraposition, we use the criteria for linearity in certain
examples in order to determine $\KO{G}$ or to even show $\KO{G}=G$.
See~\ref{ex:Burnside} but also~\ref{Sl2}, \ref{SL2timesH},
\ref{exa:monotheticConnCA}.

\begin{theo}[Mal{\cprime}tsev~\cite{MR0013164}, see~\cite{MR0034396}]
  \label{maltsev}
  Let $G$ be a connected Lie group.
  \begin{enumerate}
  \item The group $G$ is linear if, and only if, its solvable radical
    and its maximal semisimple subgroups are linear.
  \item If $G$ is semisimple and linear and $Z$ is a discrete normal
    (i.e., central) subgroup of~$G$ then $G/Z$ is linear, as well.
  \item\label{maltsev3} If $G$ is solvable then it is linear if, and
    only if, it is the semidirect product of a maximal compact
    subgroup and a simply connected normal subgroup. %
    \qed
  \end{enumerate}
\end{theo}

\begin{theo}[Nahlus~\cite{MR1396990}]\label{nahlus}
  Let $G$ be a connected Lie group with Lie algebra
  $\mathfrak{g}:=\lie{G}$, let\/~$\mathfrak{r}$ be the solvable radical
  of the commutator algebra\/~$\mathfrak{g}'$, and let\/~$\mathfrak{z}$ be
  the center of\/~$\mathfrak{g}$. Choose a maximal torus $T$ of the
  solvable radical of~$G$ and a maximal
  semisimple subgroup~$S$. Then $G$ is linear precisely if\/
  $\mathfrak{r}\cap\mathfrak{z}\cap\lie{T}=\{0\}$ and $S$ is linear. 
\qed
\end{theo}

We have seen that the structure of~$\KO{G}$ for a connected Lie
group~$G$ may depend essentially on the choice of $G$ as a quotient of
its simply connected covering~$\tilde{G}$. This raises the problem of
characterizing those Lie algebras whose associated Lie groups are
\emph{all} linear.

\begin{theo}[Moskowitz~\cite{MR0327979}]
  Let\/~$\mathfrak{g}$ be a Lie algebra of finite dimension
  over~$\RR$, let\/~$\mathfrak{r}$ be the solvable radical of the
  commutator algebra\/~$\mathfrak{g}'$, let\/~$\mathfrak{z}$ be the
  center of\/~$\mathfrak{g}$, and let\/~$\mathfrak{s}$ be a maximal
  semisimple subalgebra of\/~$\mathfrak{g}$. Then \emph{every}
  connected Lie group with Lie algebra~$\mathfrak{g}$ is linear
  precisely if the simply connected group associated to~$\mathfrak{s}$
  is linear and\/~$\mathfrak{r}\cap\mathfrak{z}=\{0\}$.  
\qed
\end{theo}

\begin{exas}\label{simpleLieInfinitePi}
  Among the connected Lie groups with simple Lie algebra, the most
  obvious non-linear examples are those with infinite center. These
  are precisely those connected Lie groups with simple Lie algebra
  where the maximal compact subgroups have centralizers of positive
  dimension.  Cf.~\cite[Ch.\,VIII, \S\,6; Ch.\,X, \S\,6]{MR514561}.
\end{exas}

\begin{theo}[Lee and Wu \cite{MR858294}]
  Assume that $G$ is a connected Lie group, and that $G$ is
  linear. Then the holomorph $\Aut{G}\ltimes G$ is linear if, and only
  if, one of the following holds:
  \begin{enumerate}
  \item The nilradical of $G$ is simply connected.
  \item The group $G$ is perfect (i.e., coincides with its commutator
    subgroup).
  \item The quotient $G/G'$ is isomorphic to $\RR/\ZZ$. 
\qed
  \end{enumerate}
\end{theo}

\goodbreak
\begin{theo}[Burnside~\cite{JFM36.0199.02}, {cf.~\cite[Ch.\,2, 2.1,
  Cor.\,B and~C, pp.\,138\,f]{MR1039816}}]\label{Burnside}\quad
  \begin{enumerate}
  \item\label{BurnsideBounded} If\/ $G\le\GL[n]{\CC}$ has finite
    exponent (i.e. if there exists $m\ge1$ such that
    $\smallset{g^m}{g\in G}$ is trivial) then $G$ is a finite group.
  \item\label{BurnsideClasses} If $H\le\GL[n]{\CC}$ contains only
    finitely many conjugacy classes then $H$ is finite.  
\qed
  \end{enumerate}
\end{theo}

\begin{theo}[Schur~\cite{42.0155.01}]\label{schurTorsion}
  If\/ $G\le\GL[n]{\CC}$ is a torsion group
  (i.e., if every element of\/~$G$ has finite order) then every
  finitely generated subgroup of\/~$G$ is finite.
\qed
\end{theo}

\begin{exas}\label{ex:Burnside}
  Using Burnside's results we provide examples of (discrete) simple
  non-abelian groups in $\KO{\Lie}$ or even in~$\KO{\SepLie}$ ---
  separability just means countability here.
  \begin{enumerate}
  \item\label{bounded} For each infinite field $F$ of positive
    characteristic~$p$ the simple group $\PSL[2]{F}$ does not admit
    any non-trivial ordinary representation because it contains an
    infinite group of exponent~$p$,
    cf.~\ref{Burnside}\ref{BurnsideBounded}. Thus $\KO{\PSL[2]{F}} =
    \PSL[2]{F}$.
  \item\label{HNN} There exists a countably infinite group $H$ such
    that every non-trivial element of~$H$ has infinite order, and all
    these elements form a single conjugacy class,
    see~\cite{MR0032641}. Clearly this group $H$ is simple, and
    $\KO{H}=H$ follows from~\ref{Burnside}\ref{BurnsideClasses}.
  \item If~$p$ is a sufficiently large prime then there exists a
    countably infinite group $\OlM{p}$ such that every proper subgroup
    of~$\OlM{p}$ has order~$p$ and any such subgroup contains a set of
    representatives for the conjugacy classes,
    see~\cite[\S\,19]{MR1191619}. We may conclude $\KO{\OlM{p}}=\OlM{p}$ from
    any one of Burnside's results as stated in~\ref{Burnside}, and
    also from~\ref{schurTorsion}.
  \end{enumerate}
\end{exas}

%%%%%%%%%%%%%%%%%%%%%%%%%%%%%%%%%%%%%%%%%%%%%%%%%%%%%%%%%%%%%%%%%%%%%%%%
%\newpage%%
\bigbreak
\goodbreak
%%%%%%%%%%%%%%%%%%%%%%%%%%%%%%%%%%%%%%%%%%%%%%%%%%%%%%%%%%%%%%%%%%%%%%%%%%%%%
% \bibliographystyle{mybibstyle}                  %%%%% %%%%% %%%%% %%%%% %%%%%
% \bibliography{myBibliography}                   %%%%% %%%%% %%%%% %%%%% %%%%%
%%%%%%%%%%%%%%%%%%%%%%%%%%%%%%%%%%%%%%%%%%%%%%%%%%%%%%%%%%%%%%%%%%%%%%%%%%%%%
\def\cprime{$'$}

%%%%%%%%%%%%%%%%%%%%%%%%%%%%%%%%%%%%%%%%%%%%%%%%%%%%%%%%%%%%%%%%%%%%%%%%%%%%%
\begin{ackn}
  A substantial part of these notes was written while the author was a
  guest of  SFB 478 ``Geometrische Strukturen
  in der Mathematik'', M\"unster, Germany.
\end{ackn}
%%%%%%%%%%%%%%%%%%%%%%%%%%%%%%%%%%%%%%%%%%%%%%%%%%%%%%%%%%%%%%%%%%%%%%%%
\begin{small}
  {\bfseries Author's address: }\\
  Markus Stroppel, Fachbereich Mathematik, Universit\"at Stuttgart,
  D-70550 Stuttgart, Germany.
\end{small}

\end{document}